\documentclass[12pt, oneside]{amsart}

\usepackage[utf8]{inputenc}
\usepackage{lmodern}

\usepackage{amssymb}
\usepackage{amsthm}
\usepackage{mathtools}
\usepackage[only,Yup]{stmaryrd} 
\usepackage{mathrsfs}    
\usepackage[new]{old-arrows}   

\usepackage[a4paper,left=2.5cm,top=3cm,right=2.5cm,bottom=2.5cm]{geometry}

\usepackage[all,cmtip]{xy}

\usepackage{enumerate}

\usepackage{multicol}

\usepackage{indentfirst}

\usepackage{hyperref}
\hypersetup{colorlinks=true,allcolors=black}  

\DeclareMathOperator{\diam}{diam}
\DeclareMathOperator{\codim}{codim}
\DeclareMathOperator{\bas}{bas}
\DeclareMathOperator{\hor}{Hor}

\DeclareMathOperator{\vol}{vol}

\DeclareMathOperator{\ric}{Ric}

\DeclareMathOperator*{\spannn}{span}
\DeclareMathOperator{\sym}{S}
\DeclareMathOperator{\poin}{P}
\DeclareMathOperator{\depth}{depth}

\mathchardef\ordinarycolon\mathcode`\:
\mathcode`\:=\string"8000
\begingroup \catcode`\:=\active
  \gdef:{\mathrel{\mathop\ordinarycolon}}
\endgroup

\numberwithin{equation}{section}

\begin{document}

\title{Equivariant basic cohomology under deformations}

\author{Francisco C.~Caramello Jr.}
\address{Departamento de Matemática, Universidade Federal de Santa Catarina, R. Eng. Agr. Andrei Cristian Ferreira, 88040-900, Florianópolis - SC, Brazil}
\email{francisco.caramello@ufsc.br}
\thanks{The first author was supported by grant \#2018/14980-0, São Paulo Research Foundation (FAPESP)}

\author{Dirk Töben}
\address{Departamento de Matemática, Universidade Federal de São Carlos, Rod.~Washington Luís, Km 235, 13565-905, São Carlos - SP, Brazil}
\email{dirktoben@dm.ufscar.br}


\theoremstyle{definition}
\newtheorem{example}{Example}[section]
\newtheorem{definition}[example]{Definition}
\newtheorem{remark}[example]{Remark}
\theoremstyle{plain}
\newtheorem{proposition}[example]{Proposition}
\newtheorem{theorem}[example]{Theorem}
\newtheorem{lemma}[example]{Lemma}
\newtheorem{corollary}[example]{Corollary}
\newtheorem{claim}[example]{Claim}
\newtheorem{scholium}[example]{Scholium}
\newtheorem{thmx}{Theorem}
\newtheorem{corx}[thmx]{Corollary}

\newcommand{\dif}[0]{\mathrm{d}}
\newcommand{\od}[2]{\frac{\dif #1}{\dif #2}}
\newcommand{\pd}[2]{\frac{\partial #1}{\partial #2}}
\newcommand{\dcov}[2]{\frac{\nabla #1}{\dif #2}}
\newcommand{\proin}[2]{\left\langle #1, #2 \right\rangle}
\newcommand{\f}[0]{\mathcal{F}}
\newcommand{\g}[0]{\mathcal{G}}
\newcommand{\piorb}[0]{\pi_1^{\mathrm{orb}}}
\newcommand\blfootnote[1]{%
  \begingroup
  \renewcommand\thefootnote{}\footnote{#1}%
  \addtocounter{footnote}{-1}%
  \endgroup
}

\newcommand{\cm}[1]{\marginpar{*}{\tiny #1 }}
\newcommand{\bt}{\tiny \bf}

\begin{abstract}
There is a natural way to deform a Killing foliation with non-closed leaves, due to Ghys and Haefliger--Salem, into a closed foliation, i.e., a foliation whose leaves are all closed. Certain transverse geometric and topological properties are preserved under these deformations, as previously shown by the authors. For instance, the basic Euler characteristic is invariant. In this article we show that the equivariant basic cohomology ring structure is preserved under these deformations, which in turn leads to a sufficient algebraic condition (namely, equivariant formality) for the Betti numbers of basic cohomology to be preserved as well. In particular, this is true for the deformation of the Reeb orbit foliation of a $K$-contact manifold. Another consequence is that there is a universal bound on the sum of basic Betti numbers of any equivariantly formal, positively curved Killing foliation of a given codimension. We also show that a Killing foliation with negative transverse Ricci curvature is closed. If the transverse sectional curvature is negative we show, furthermore, that its fundamental group has exponential growth. Finally, we obtain a transverse generalization of Synge's theorem to Killing foliations.
\end{abstract}

\maketitle
\setcounter{tocdepth}{1}
\tableofcontents

\section{Introduction}
A Killing foliation is a complete Riemannian foliation whose Molino sheaf is trivial (see Section \ref{section: killing foliations}). This class includes complete Riemannian foliations of simply connected manifolds and also foliations given by the orbits of isometric actions. A construction by A.~Haefliger and E.~Salem in \cite{haefliger2} (see also \cite[Théorème A]{ghys} by E.~Ghys) allows one to deform a non-closed Killing foliation $\f$ into a \emph{closed} foliation $\g$, that is, a foliation whose leaves are all closed submanifolds, which can be chosen arbitrarily close to $\f$. Here we will call this a \textit{regular deformation} of $\f$. In \cite[Theorem B]{caramello}, which we restate here as Theorem \ref{theorem: deformation}, we showed that some aspects of the transverse geometry, like the sign of transverse sectional curvature, are preserved throughout those deformations. Results from Riemannian geometry could then be applied to the orbifold quotient $M/\g$, from which we could then draw back conclusions about the structure of the original foliation $\f$ (whose quotient is not accessible to usual tools of Riemannian geometry).

The transverse topology of $\f$ can also be studied using this deformation technique. The complex of basic forms $H(\f)$, for example, plays the role of the de Rham cohomology of the ``virtual'' quotient $M/\f$ (see Section \ref{subsection: foliations}). We have shown in \cite[Theorem 7.4]{caramello} that the basic Euler characteristic $\chi(\f)$ remains constant under regular deformations. In spite of this, the basic Betti numbers are in general not invariant (see Example \ref{exe: nozawa}). In this paper we show that the ring structure of the {\em equivariant} basic cohomology is in fact preserved under regular deformations. More specifically, a Killing foliation $\f$ admits a natural transverse infinitesimal action of an Abelian Lie algebra $\mathfrak{a}$ (its structural algebra, see Section \ref{section: killing foliations}), which allows one to define the $\mathfrak{a}$-equivariant cohomology $H_{\mathfrak{a}}(\f)$. This is a transverse generalization of the classical equivariant cohomology with respect to a torus action on a manifold. It was first introduced in \cite{goertsches} and further studied in \cite{goertsches2} and \cite{toben}. Here we show the following.

\begin{thmx}\label{thrm: invariance of H_a intro}
Let $\f$ be a Killing foliation of a compact manifold $M$ and let $\f_t$ be a regular deformation of $\f$. Then, for each $t$,
$$H_{\mathfrak{a}}(\f)\cong H_{\mathfrak{a}}(\f_t)$$
as $\mathbb{R}$-algebras.
\end{thmx}

As a corollary we get that $H_\mathfrak{a}(\f)$ is isomorphic to the $\mathbb{T}^d$-equivariant cohomology of the quotient orbifold $M/\g$ of the closed foliation $\g=\f_1$, by an application of the equivariant De Rham theorem for orbifolds (Theorem \ref{thrm: equivariant De Rham theorem for orbifolds}). Another application of Theorem \ref{thrm: invariance of H_a intro} yields us that the equivariant formality of the $\mathfrak{a}$-action on $\f$ is a sufficient condition for the basic Betti numbers to remain constant throughout regular deformations.

\begin{thmx}\label{theo: Betti numbers intro}
Let $\f$ be a Killing foliation of a compact manifold $M$ and let $\f_t$ be a regular deformation. If the transverse action of the structural algebra $\mathfrak{a}$ of $\f$ is equivariantly formal, then $b_i(\f_t)$ is constant in $t$, for each $i$.
\end{thmx}

There are several interesting situations in which equivariant formality of the $\mathfrak{a}$-action holds:

\begin{corx}
Let $\f$ be a transversely orientable Killing foliation of a compact manifold $M$ and let $M^{\mathfrak{a}}$ be the union of closed leaves. If $H^{odd}(M,\f)=0$, if $\dim H(M^{\mathfrak{a}}/\f)=\dim H(\f)$, or if there is a basic Morse--Bott function whose critical set is equal to $M^{\mathfrak{a}}$, then the basic Betti numbers are constant under regular deformations.
\end{corx}

A significant particular case of the later is that of compact $K$-contact manifolds, which includes compact Sasakian manifolds. In fact, the Reeb orbit foliation is a homogeneous Riemannian foliation and therefore a Killing foliation. It is defined by a one-parameter subgroup of the isometry group, whose closure is hence a torus. Regular deformations are exactly families of one-parameter subgroups in this torus. The contact moment map is a basic Morse-Bott function whose critical set is the union of closed leaves (see \cite[Proposition 6.3]{goertsches3}), so we have the following.

\begin{corx}
Let $(M, \alpha, \mathrm{g})$ be a compact $K$-contact manifold. Then the basic Betti numbers of the Reeb orbit foliation are constant under regular deformations.
\end{corx}

We also investigate in this article some applications of the deformation technique to the transverse geometry of $\f$, including some results on how the sign of transverse sectional and Ricci curvatures ($\sec_\f$ and $\ric_\f$, respectively) influences the transverse topology of $\f$. For instance, a classical theorem by Gromov establishes that there is a constant $C=C(n)$ that bounds the sum of all Betti numbers of any non-negatively curved Riemannian manifold of dimension $n$ (see \cite[\S 0.2A]{gromov2}). There is a generalization of this theorem for Alexandrov spaces \cite[Theorem 1]{koh} which can be applied, in particular, to orbifolds. A consequence of Theorem \ref{theo: Betti numbers intro} is that, in the equivariantly formal case, the basic Betti numbers of $\f$ coincide with the Betti numbers of the orbifold $M/\g$, for a closed approximation $\g$ (see Corollary \ref{cor: basic betti numbers as orbifold}). This gives us the following.

\begin{thmx}\label{theorem: Gromov intro}
There exists a constant $C=C(q)$ such that every $q$-codimensional Killing foliation $\f$ of a compact manifold $M$ with $\sec_\f>0$ and whose transverse action of the structural algebra $\mathfrak{a}$ is equivariantly formal satisfies
$$\sum_{i=0}^q b_i(\f) \leq C.$$
\end{thmx}

This prompts the question of whether the hypothesis that the action is equivariantly formal can be dropped. We also prove an orbifold version of Bochner's theorem on Killing vector fields (Theorem \ref{teo: bochner for orbifolds}) and apply it to the leaf space $M/\g$ of a closed approximation of $\f$ to obtain the theorem below.

\begin{thmx}\label{theorem: bochner intro}
Let $(M,\f)$ be a complete Riemannian foliation with transverse Ricci curvature satisfying $\ric_\f\leq c <0$. If either
\begin{enumerate}[(i)]
\item $\f$ is a Killing foliation and $M$ is compact, or
\item $\f$ is transversely compact and $|\pi_1(M)|<\infty$,
\end{enumerate}
then $\f$ is closed.
\end{thmx}

This result should be compared with \cite[Theorem 2]{hebda} by J. Hebda, which establishes that a complete Riemannian foliation $\f$ with transverse \emph{sectional} curvature satisfying $\sec_\f\leq0$ is developable, that is, its lift to the universal covering of $M$ is given by the fibers of a submersion $\widetilde{M}\to N$. Hebda also proves in \cite{hebda} that a compact manifold whose fundamental group is nilpotent does not admit a Riemannian foliation with $\sec_\f<0$. Here we show the following.

\begin{thmx}\label{theo: Milnor trasnverso intro}
Let $\f$ be a Killing foliation on a compact manifold $M$ such that $\sec_\f<0$. Then $\f$ is closed and $\pi_1(\f)$ grows exponentially. In particular, $\pi_1(M)$ grows exponentially.
\end{thmx}

This is a transverse generalization of Milnor's theorem on the growth of the fundamental group \cite[Theorem 2]{milnor}. Here $\pi_1(\f)$ is the fundamental group of the holonomy pseudogroup of $\f$, which coincides with $\pi_1(M)$ when $\f$ is the trivial foliation by points (and, more generally, with $\piorb(M/\f)$, when $\f$ is closed). Finally, combining the deformation technique with the orbifold version of Synge's theorem, proved in \cite[Corollary 2.3.6]{yeroshkin}, we obtain the following transverse version of this theorem.

\begin{thmx}\label{theorem: synge intro}
Let $\f$ be a Killing foliation of a compact manifold $M$, with $\sec_\f >0$.
\begin{enumerate}[(i)]
\item If $\codim\f$ is even and $\f$ is transversely orientable, then $M/\overline{\f}$ is simply connected.
\item If $\codim\f$ is odd and, for each $L\in \f$, the germinal holonomy of $L$ preserves transverse orientation, then $\f$ is transversely orientable.
\end{enumerate}
\end{thmx}

\section{Preliminaries}\label{section: preliminaires}

In this section we establish our notation and review some basic notions on Riemannian foliations. Throughout this article the objects are supposed to be of class $C^\infty$ (smooth).

\subsection{Foliations}\label{subsection: foliations}

Let $\f$ denote a $p$-dimensional foliation of a $(p+q)$-dimensional \emph{connected} manifold $M$ without boundary. The number $q$ is the \textit{codimension} of $\f$. The subbundle of $TM$ consisting of the subspaces tangent to the leaves will be denoted by $T\f$ and the Lie algebra of the vector fields with values in $T\f$ by $\mathfrak{X}(\f)$. We say that $\f$ is \textit{transversely orientable} when its \textit{normal bundle} $\nu\f:=TM/T\f$ is orientable. The set of the closures of the leaves of $\f$ is denoted by $\overline{\f}:=\{\overline{L}\ |\ L\in\f\}$. In the simple case where all leaves are closed, i.e. $\overline{\f}=\f$, we say that $\f$ is a \textit{closed} foliation. Furthermore, we say that $\f$ is \textit{transversely compact} when $M/\overline{\f}$ is compact.

Recall that $\f$ can be defined by an open cover $\{U_i\}_{i\in I}$ of $M$, submersions $\pi_i:U_i\to T_i$, with $T_i\subset\mathbb{R}^q$, and diffeomorphisms $\gamma_{ij}:\pi_j(U_i\cap U_j)\to\pi_i(U_i\cap U_j)$ satisfying $\gamma_{ij}\circ\pi_j|_{U_i\cap U_j}=\pi_i|_{U_i\cap U_j}$ for all $i,j\in I$. The collection $\{\gamma_{ij}\}$ is a \textit{Haefliger cocycle} representing $\f$. We assume without loss of generality that the fibers  of each $\pi_i$ are connected. The pseudogroup of local diffeomorphisms generated by $\gamma=\{\gamma_{ij}\}$ acting on $T_\gamma:=\bigsqcup_i T_i$ is the \textit{holonomy pseudogroup} of $\f$ associated to $\gamma$, that we denote by $\mathscr{H}_\gamma$. If $\delta$ is another cocycle defining $\f$ then $\mathscr{H}_\delta$ is equivalent to $\mathscr{H}_\gamma$, meaning that there is a maximal collection $\Phi$ of diffeomorphisms $\varphi$ from open sets of $T_\delta$ to open sets of $T_\gamma$ such that $\{\mathrm{Dom}(\varphi)\ |\ \varphi\in\Phi\}$ covers $T_\delta$, $\{\mathrm{Im}(\varphi)\ |\ \varphi\in\Phi\}$ covers $T_\gamma$ and, for all $\varphi,\psi\in\Phi$, $h\in\mathscr{H}_\delta$ and $h'\in\mathscr{H}_\gamma$, one has $\psi^{-1}\circ h'\circ\varphi\in\mathscr{H}_\delta$, $\psi\circ h\circ\varphi^{-1}\in\mathscr{H}_\gamma$ and $h'\circ\varphi\circ h\in\Phi$. We write $(T_\f,\mathscr{H}_\f)$ to denote a representative of the equivalence class of these pseudogroups. Note that the orbit space $T_\f/\mathscr{H}_\f$ coincides with the leaf space $M/\f$. For a leaf $L\in\f$, we denote the \textit{germinal holonomy group of $L$ at $x\in L\cap T_\f$} by $\mathrm{Hol}_x(L)$, which is the group of germs of elements in the stabilizer $(\mathscr{H}_\f)_x$ (for details, see \cite[\S 2.3]{candel} or \cite[Section 2.1]{mrcun}).

There is a generalization of the notion of fundamental group to pseudogroups, which is defined in terms of homotopy classes of $\mathscr{H}$-loops. We refer to \cite{salem} for details. If we have a foliation $(M,\f)$ and fix $x\in M$ and $\overline{x}=\pi_i(x)\in T_\f$, this construction applied to $\mathscr{H}_\f$ furnishes an invariant $\pi_1(\f,\overline{x})$ of $\f$, whose isomorphism class does not depend on the Haefliger cocycle representing $\f$. Moreover, when $M$ is connected this invariant is independent of the base point, up to conjugacy, so we will often denote it simply by $\pi_1(\f)$. There is a natural surjection
\begin{equation}\pi_1(M,x)\longrightarrow\pi_1(\f,\overline{x})\label{surjection fundamental groups}\end{equation} defined by locally projecting a loop on $M$ by a collection of submersions $\pi_i$, whose domains cover the image of the loop, and gluing those projections by the appropriate holonomy maps $\gamma_{ij}$ (see \cite[Section 1.11]{salem}). 

A \textit{foliate field} on $M$ is a vector field in the Lie subalgebra
$$\mathfrak{L}(\f)=\{X\in\mathfrak{X}(M)\ |\ [X,\mathfrak{X}(\f)]\subset\mathfrak{X}(\f)\}.$$
If $X\in\mathfrak{L}(\f)$ and $\pi:U\to T$ is a submersion locally defining $\f$, then $X|_U$ is $\pi$-related to some vector field $X_T\in\mathfrak{X}(T)$. In fact, this characterizes $\mathfrak{L}(\f)$ \cite[Section 2.2]{molino}. The Lie algebra $\mathfrak{L}(\f)$ also has the structure of a module over the ring $\Omega^0(\f)$ of \textit{basic functions} of $\f$, that is, functions $f\in C^{\infty}(M)$ such that $Xf=0$ for every $X\in\mathfrak{X}(\f)$. The quotient of $\mathfrak{L}(\f)$ by the ideal $\mathfrak{X}(\f)$ yields the Lie algebra $\mathfrak{l}(\f)$ of \textit{transverse vector fields}. For $X\in\mathfrak{L}(\f)$, we denote its induced transverse field by $\overline{X}\in\mathfrak{l}(\f)$. Note that each $\overline{X}$ defines a unique section of $\nu\f$ and corresponds to a unique $\mathscr{H}_\f$-invariant vector field on a chosen transversal $T_\f$.

A (covariant) tensor field $\xi$ on $M$ is \textit{$\f$-basic} if its invariant, i.e.~$\mathcal{L}_X\xi=0$ for all $X\in\mathfrak{X}(\f)$, and horizontal, i.e.~$\xi(X_1,\dots,X_k)=0$ whenever some $X_i\in\mathfrak{X}(\f)$. These are the tensor fields that project through the local defining submersions to $\mathscr{H}_\f$-invariant tensor fields on $T_\f$. We denote the algebra of $\f$-basic tensor fields by $\mathcal{T}(\f)$. In particular, we can consider the algebra $\Omega(\f)<\Omega(M)$ of \textit{$\f$-basic differential forms}. By Cartan's formula, $\alpha$ is basic if, and only if, $i_X\alpha=0$ and $i_X(d\alpha)=0$ for all $X\in\mathfrak{X}(\f)$. Hence $\Omega(\f)$ is closed under the exterior derivative. The cohomology groups of the subcomplex
$$\cdots \stackrel{d}{\longrightarrow} \Omega^{i-1}(\f) \stackrel{d}{\longrightarrow} \Omega^i(\f) \stackrel{d}{\longrightarrow} \Omega^{i+1}(\f) \stackrel{d}{\longrightarrow}\cdots ,$$
are the \textit{basic cohomology groups} of $\f$, denoted by $H^i(\f)$. If the dimensions $\dim(H^i(\f))$ are all finite, we define the \textit{basic Euler characteristic} of $\f$ as the alternating sum
$$\chi(\f)=\sum_i(-1)^i\dim(H^i(\f)).$$
In analogy with the classical case, we call $b_i(\f):=\dim(H^i(\f))$ the \textit{basic Betti numbers} of $\f$. Notice that when $\f$ is the trivial foliation of $M$ by points this recovers the usual Euler characteristic and the Betti numbers of $M$.

\subsection{Orbifolds}

In this section we will briefly recall some facts about orbifolds and adapt some classical theorems from Riemannian geometry to this setting. We refer to \cite[Chapter 1]{adem}, \cite[Section 2.4]{mrcun} and \cite[Section 2]{kleiner} for more detailed introductions to orbifolds.

Let $\mathcal{O}$ be a smooth orbifold and $\mathcal{A}=\{(\widetilde{U}_i,H_i,\phi_i)\}_{i\in I}$ an atlas for $\mathcal{O}$. The underlying topological space of $\mathcal{O}$ will be denoted by $|\mathcal{O}|$. Recall that each chart of $\mathcal{A}$ consists of a connected open subset $\widetilde{U}$ of $\mathbb{R}^n$, a finite subgroup $H$ of $\mathrm{Diff}(\widetilde{U})$ and an $H$-invariant map $\phi:\widetilde{U}\to |\mathcal{O}|$ that descends to a homeomorphism $\widetilde{U}/H\cong U$, for an open subset $U\subset |\mathcal{O}|$. Given $(\widetilde{U},H,\phi)$ and $x=\phi(\tilde{x})\in U$, the \textit{local group} $\Gamma_x$ of $\mathcal{O}$ at $x$ is the stabilizer $H_{\tilde{x}}<H$. Its isomorphism class is independent of both the chart and the point $\tilde{x}$ over $x$. Note that for every $x\in |\mathcal{O}|$ it is always possible to choose a chart $(\widetilde{U}_x,\Gamma_x,\phi_x)$ over $x$. We say that $\mathcal{O}$ is \textit{locally orientable} when each $\Gamma_x$ acts by orientation-preserving diffeomorphisms of $\widetilde{U}_x$. The \textit{canonical stratification} of $\mathcal{O}$ is the decomposition
$$|\mathcal{O}|=\bigsqcup \Sigma_\alpha,$$
where each $\Sigma_\alpha$ is a connected component of some $\Sigma_\Gamma=\{x\in|\mathcal{O}|\ |\ \Gamma_x\cong\Gamma\}$. Each $\Sigma_\alpha$ is a manifold. Consider, for each $i$, the union $\Sigma^i$ of all $\Sigma_\alpha$ of dimension $i$, and let $r_1<\dots<r_k=n$ be the indices for which $\Sigma^i\neq\emptyset$. The \textit{regular stratum} $\mathcal{O}_{\mathrm{reg}}:=\Sigma^n$ is open and dense in $|\mathcal{O}|$. The \textit{singular locus} of $\mathcal{O}$ is the set $|\mathcal{O}|\setminus\mathcal{O}_{\mathrm{reg}}$. It will also be useful to consider $\mathcal{O}_{\mathrm{min}}:=\Sigma^{r_1}$.

\begin{example}\label{example: orbifold structure on the leaf space} If $\f$ is a smooth foliation of codimension $q$ of a manifold $M$ and the leaves of $\f$ are compact and with finite holonomy, then $M/\f$ has a natural $q$-dimensional orbifold structure relative to which the local group of a leaf in $M/\f$ is its holonomy group \cite[Theorem 2.15]{mrcun}. We will denote the orbifold obtained this way by $M/\!/\f$ in order to distinguish it from the topological space $M/\f$. Of course, $|M/\!/\f|=M/\f$. Analogously, if a compact Lie group $G$ acts on $M$ and $\dim(G_x)$ is a constant function (that is, if the action is \textit{foliated}), then the quotient orbifold will be denoted by $M/\!/G$ (see also Example \ref{exe: foliated actions}). A foliation given by a foliated action is said to be \textit{homogeneous}. Conversely, any orbifold arises as the quotient space of a foliated action of a compact connected Lie group on a manifold. More precisely, $\mathcal{O}$ is diffeomorphic to $\mathcal{O}^{\Yup}_{\mathbb{C}}/\!/\mathrm{U}(n)$, where $\mathcal{O}^{\Yup}_{\mathbb{C}}$ is the unitary frame bundle of $\mathcal{O}$ with respect to some chosen complex Riemannian structure \cite[Proposition 2.23]{mrcun}.
\end{example}

Consider $U_\mathcal{A}:=\bigsqcup_{i\in I}\widetilde{U}_i$ and $\phi:=\{\phi_i\}_{i\in I}:U_\mathcal{A}\to |\mathcal{O}|$, that is, $x\in \widetilde{U}_i\subset U_\mathcal{A}$ implies $\phi(x)=\phi_i(x)$. A \textit{change of charts} of $\mathcal{A}$ is a diffeomorphism $h:V\to W$, with $V,W\subset U_\mathcal{A}$ open sets, such that $\phi\circ h=\phi|_V$. The collection of all changes of charts of $\mathcal{A}$ form a pseudogroup $\mathscr{H}_{\mathcal{A}}$ of local diffeomorphisms of $U_\mathcal{A}$, and $\phi$ induces a homeomorphism $U_\mathcal{A}/\mathscr{H}_{\mathcal{A}}\to|\mathcal{O}|$. As in the case of the holonomy pseudogroup of a foliation, if $\mathcal{B}$ is another atlas that is compatible with $\mathcal{A}$, the corresponding pseudogroups are equivalent. We will denote by $(U_\mathcal{O},\mathscr{H}_{\mathcal{O}})$ a representative of the equivalence class of these pseudogroups.


A \textit{smooth map} $f:\mathcal{O}\to\mathcal{P}$ consists of a continuous map $|f|:|\mathcal{O}|\to|\mathcal{P}|$ such that, for every $x\in|\mathcal{O}|$, there are charts $(\widetilde{U},H,\phi)$ and $(\widetilde{V},K,\psi)$ around $x$ and $f(x)$, respectively, a homomorphism $\overline{f}_x:H\to K$ and an $\overline{f}_x$-equivariant smooth map $\tilde{f}_x:\widetilde{U}\to\widetilde{V}$ satisfying $f(U)\subset V$ and $\psi\circ\tilde{f}_x=f\circ\phi$. There are relevant refinements of this notion, such as the \textit{good maps} defined in \cite{chenruan}, that are needed in some elementary constructions. These good maps correspond to morphisms when the orbifolds are viewed as Lie groupoids \cite[Proposition 5.1.7]{lupercio}. In particular, a smooth map $M\to\mathcal{O}$ ``in the orbifold sense'', as defined in \cite[p. 715]{haefliger2}, is a good map. Diffeomorphisms of orbifolds are smooth maps with smooth inverses. In particular, diffeomorphisms preserve the canonical stratification by local groups.

For $x\in|\mathcal{O}|$ fixed, we define the \textit{fundamental group of $\mathcal{O}$} at $x$ as $\piorb(\mathcal{O},x):=\pi_1(\mathscr{H}_{\mathcal{O}},\tilde{x})$, for some lift $\tilde{x}$ of $x$. This definition of $\piorb$ is equivalent to Thurston's definition via the automorphism group of the orbifold universal covering $\widetilde{\mathcal{O}}\to\mathcal{O}$ (see, for instance, \cite[Corollary 3.19]{bridson}). When $\mathcal{O}$ is connected, the isomorphism class of $\piorb(\mathcal{O},x)$ does not depend on $x$, so we often omit it in the notation. The group $\piorb(\mathcal{O})$ differs from $\pi_1(|\mathcal{O}|)$ in general, for it also captures information on the singularities.

Analogously to manifolds, one can define orbibundles over orbifolds (see, for instance \cite[p. 7]{kleiner}). The tangent bundle $T\mathcal{O}$, for example, is locally diffeomorphic to $T\widetilde{U}_x/\!/H_x$, for a chart $(\widetilde{U}_x,H_x,\phi_x)$. We follow \cite{kleiner} and consider the \textit{tangent space} $T_x\mathcal{O}$ to be the orbivector space isomorphic to $T_x\widetilde{U}_x$ together with the linearized $H_x$-action, while $C_x|\mathcal{O}|:=T_x\widetilde{U}_x/H_x$ is the \textit{tangent cone} at $x$. One then carry over the definition of the usual objects from the differential topology/geometry of manifolds, such as differential forms and Riemannian metrics, to the  orbifold setting as sections of appropriate orbibundles. These objects will then correspond to $\mathscr{H}_{\mathcal{O}}$-invariant objects on $U_\mathcal{O}$. A smooth differential form on $\mathcal{O}$, for example, is an $\mathscr{H}_{\mathcal{O}}$-invariant differential form on $U_\mathcal{O}$. The complex of smooth differential forms on $\mathcal{O}$ will be denoted by $\Omega(\mathcal{O})$, and its cohomology by $H(\mathcal{O})$. The following result can be seen as an orbifold version of De Rham's Theorem (see \cite[Theorem 1]{satake} or \cite[Theorem 2.13]{adem}).

\begin{theorem}[{\cite[Theorem 1]{satake}}]\label{theorem: Satake}
Let $\mathcal{O}$ be an orbifold. Then $H^i(\mathcal{O})\cong H^i(|\mathcal{O}|,\mathbb{R})$.
\end{theorem}

Since, for a foliation $(M,\f)$, there is an identification between $\f$-basic forms and $\mathscr{H}_\f$-invariant forms on $T_\f$, the following result is clear.

\begin{proposition}\label{prop: basic cohomology of closed foliations}
Let $(M,\f)$ be a smooth foliation whose leaves are all compact and have finite holonomy. Then the projection $\pi:M\to M/\!/\f$ induces an isomorphism of differential complexes $\pi^*:\Omega(M/\!/\f)\to\Omega(\f)$. In particular, $H(\f)\cong H(M/\!/\f)$.
\end{proposition}

A smooth \textit{foliation of an orbifold} $\mathcal{O}$ is a smooth foliation $\f$ of $U_\mathcal{O}$, which is $\mathscr{H}_{\mathcal{O}}$-invariant. The atlas can be chosen so that on each $\widetilde{U}_i$ the foliation is defined by the connected fibers of a submersion $\pi_i$ onto a manifold $T_i$. In this case the holonomy pseudogroup $\mathscr{H}_\f$ of $\f$ is generated by the local diffeomorphisms of $T_\f:=\bigsqcup_{i\in I}T_i$ that are the projections of elements in $\mathscr{H}_{\mathcal{O}}$ (see \cite[\S 3.2]{haefliger2}). All notions defined in Section \ref{subsection: foliations} for foliations on manifolds therefore extend to foliations on orbifolds.

\begin{example}
Suppose that a Lie group $G$ acts smoothly on an orbifold $\mathcal{O}$, that is, there is a smooth map $\mu:G\times \mathcal{O}\to \mathcal{O}$ whose underlying continuous map $|\mu|$ is an action. As in the case of actions on manifolds (see Example \ref{example: orbifold structure on the leaf space}), if $\dim(G_x)$ is constant then the action defines a foliation of $\mathcal{O}$. Notice that an action respects the canonical stratification of $\mathcal{O}$.
\end{example}


A \textit{Riemannian metric} on an orbifold $\mathcal{O}$ is a symmetric positive tensor field $\mathrm{g}\in\bigotimes_0^2(\mathcal{O})$, which corresponds to an $\mathscr{H}_{\mathcal{O}}$ -invariant Riemannian metric on $U_{\mathcal{O}}$. The Levi-Civita connection $\nabla$ on $U_\mathcal{O}$ is invariant by the changes of charts and thus can be seen as a covariant derivative on $T\mathcal{O}$. The \textit{curvature tensor} of $\mathcal{O}$ is the curvature tensor of $\nabla$ on $U_{\mathcal{O}}$. Derived curvature notions, such as sectional and Ricci curvatures (which we denote $\sec_{\mathcal{O}}$ and $\ric_{\mathcal{O}}$, respectively), are defined accordingly. On a Riemannian orbifold one can define the length of a piecewise smooth curve in the usual fashion and induce a length structure $d$ on $|\mathcal{O}|$, in complete analogy to the manifold case.

We end this section establishing orbifold generalizations of two classical theorems which will be useful later, the first one being an orbifold version of Bochner's theorem on Killing vector fields. Let $(\mathcal{O},\mathrm{g})$ be a Riemannian orbifold and $X\in\mathfrak{iso}(\mathcal{O})=\{X\in\mathfrak{X}\ |\ \mathcal{L}_X\mathrm{g}=0\}$ be a Killing vector field. Consider $f:=\mathrm{g}(X,X)/2=\|X\|^2/2$. Then it follows exactly as in the manifold case (see, for instance, \cite[Proposition 29, page 191]{petersen}), that
$$\Delta f=\|\nabla X\|^2-\ric_{\mathcal{O}}(X),$$
where $\|\cdot\|$ denotes the Frobenius norm. We say that $\mathcal{O}$ has \textit{quasi-negative Ricci curvature} if $\ric_{\mathcal{O}}\leq0$ and, for some point $x\in|\mathcal{O}|$, one has $\ric_{\mathcal{O}}(v)<0$ for all $v\in T_x\mathcal{O}\setminus\{0\}$.

\begin{theorem}[Bochner's theorem for orbifolds]\label{teo: bochner for orbifolds}
Let $(\mathcal{O},\mathrm{g})$ be a connected, compact, oriented Riemannian orbifold with $\ric_{\mathcal{O}}\leq0$. Then every Killing vector field on $\mathcal{O}$ is parallel, and $\dim\mathrm{Iso}(\mathcal{O})\leq \dim\mathcal{O}_{\mathrm{min}}$. Moreover, if $\mathcal{O}$ has quasi-negative Ricci curvature, then $\mathrm{Iso}(\mathcal{O})$ is finite.
\end{theorem}

\begin{proof}
The proof is analogous to that of the classical Bochner theorem for manifolds (see, e.g., \cite[Theorem 36]{petersen}). Let $X$ be a Killing vector field on $\mathcal{O}$. For $f:=\frac{1}{2}\|X\|^2$ we have, by Stokes' theorem and the hypothesis on the curvature, that
\begin{equation}\label{equation: bochner for orbifolds}0=\int_{\mathcal{O}}\Delta f dV=\int_{\mathcal{O}}\left(\|\nabla X\|^2-\ric_{\mathcal{O}}(X)\right) dV\geq\int_{\mathcal{O}}\|\nabla X\|^2dV\geq0.\end{equation}
Therefore $\|\nabla X\|\equiv 0$, hence $X$ is parallel. Now, since each Killing vector field is parallel,
$$e_x:\mathfrak{iso}(\mathcal{O})\ni X\longmapsto X_x\in T_x\mathcal{O}$$ is injective, for any $x\in\mathcal{O}$. Choose $x\in\mathcal{O}_{\mathrm{min}}$. Then $T_x\mathcal{O}=T_x\mathcal{O}_{\mathrm{min}}\oplus \nu_x\mathcal{O}_{\mathrm{min}}$, where $T_x\mathcal{O}_{\mathrm{min}}=(T_x\mathcal{O})^{\Gamma_x}$. Since $\Gamma_x$ acts non-trivially on $\nu_x\mathcal{O}_{\mathrm{min}}$, vectors in $\nu_x\mathcal{O}_{\mathrm{min}}\setminus\{0\}$ do not admit extensions to vector fields. Hence we must have $e_x(\mathfrak{iso}(\mathcal{O}))\subset T_x\mathcal{O}_{\mathrm{min}}$, and thus $\dim\mathrm{Iso}(\mathcal{O})\leq \dim\mathcal{O}_{\mathrm{min}}$.

Now assume $\mathcal{O}$ has quasi-negative Ricci curvature and $X$ is not trivial. Then, since $X$ is parallel, it vanishes nowhere. In particular, if $x\in|\mathcal{O}|$ is the point where $\ric_{\mathcal{O}}$ is negative, then $\ric_{\mathcal{O}}(X_x)<0$. But equation \eqref{equation: bochner for orbifolds} now reads 
$$\int_{\mathcal{O}}-\ric_{\mathcal{O}}(X) dV=0,$$
therefore $\ric_{\mathcal{O}}(X)\equiv 0$, a contradiction. Thus $X\equiv 0$, yielding that $\mathrm{Iso}(\mathcal{O})$ is finite, since it is a compact Lie group whose identity component is trivial.
\end{proof}

We are now interested in an orbifold version Milnor's theorem on the growth of the fundamental group of negatively curved manifolds \cite[Theorem 2]{milnor}. Recall for this that, for a finitely generated group $\Gamma=\langle S \rangle$, with $S=\{g_1,\dots,g_k\}$, the $S$-growth function of $\Gamma$ associates to $j\in\mathbb{N}$ the number $\#_S(j)$ of distinct elements of $\Gamma$ which can be expressed as words of length at most $j$ in the alphabet $\{g_1,\dots,g_k,g_1^{-1},\dots,g_k^{-1}\}$. We say that $\Gamma$ has \textit{exponential growth} when $\#_S(j)\geq \alpha^j$ for some $\alpha>1$. This property is independent of the set $S$ of generators \cite[Lemma 1]{milnor}.

\begin{theorem}[Milnor's theorem for negatively curved orbifolds]\label{theorem: Milnor growth for orbifolds}
Let $\mathcal{O}$ be a compact Riemannian orbifold with $\sec_{\mathcal{O}}<0$. Then $\piorb(\mathcal{O})$ has exponential growth.
\end{theorem}

\begin{proof}
The proof is essentially the same as Milnor's proof for the manifold case (see \cite{milnor}, cf. also \cite[Proposition 10]{borzellino4} where it is proven that $\piorb$ has polynomial growth when Ricci curvature is nonnegative). In fact, the universal covering of $\mathcal{O}$ is a Hadamard manifold $M$ on which $\piorb(\mathcal{O})$ acts isometrically and properly discontinuously (see \cite[Corollary 2.16]{bridson}). By the Švarc–Milnor lemma (see, e.g., \cite[Corollary 5.4.2]{loh}), $M$ is quasi-isometric to the Cayley graph of $\piorb(\mathcal{O})$ (with respect to any set of generators $S$) endowed with the word metric. Now one uses Günther's inequality to compare $\#_S$ to the volume growth on $M$ (see \cite[pp. 3--4]{milnor}), obtaining the result.
\end{proof}

\subsection{Riemannian and Killing foliations}\label{section: killing foliations}

A \textit{transverse metric} for a smooth foliation $(M,\f)$ is a symmetric, positive, basic $(2,0)$-tensor field $\mathrm{g}^T$ on $M$. In this case $(M,\f,\mathrm{g}^T)$ is called a \textit{Riemannian foliation}. A Riemannian metric $\mathrm{g}$ on $M$ is \textit{bundle-like} for $\f$ if for any open set $U$ and any vector fields $Y,Z\in\mathfrak{L}(\f|_U)$ that are perpendicular to the leaves, $\mathrm{g}(Y,Z)\in\Omega^0(\f|_U)$. Any bundle-like metric $\g$ determines a transverse metric by $\mathrm{g}^T(X,Y):=\mathrm{g}(X^\bot,Y^\bot)$ with respect to the decomposition $TM=T\f\oplus T\f^\perp$. Conversely, given $\mathrm{g}^T$ one can always choose a bundle-like metric on $M$ that induces it \cite[Proposition 3.3]{molino}. With a chosen bundle-like metric, we make the identification $\nu\f\equiv T\f^\perp$. A (local) transverse vector field $X$ is a \textit{transverse Killing vector field} if $\mathcal{L}_X\mathrm{g}^T=0$. The space of global $\f$-transverse Killing vector fields will be denoted by $\mathfrak{iso}(\f)$.

\begin{example}\label{exe: foliated actions}
If a foliation $\f$ on $M$ is given by a foliated action of a Lie group $G$ and $\mathrm{g}$ is a Riemannian metric on $M$ such that $G$ acts by isometries, then $\mathrm{g}$ is bundle-like for $\f$ \cite[Remark 2.7(8)]{mrcun}. In other words, a foliation induced by an isometric action is Riemannian.
\end{example}

It follows from the definition that $\mathrm{g}^T$ projects to Riemannian metrics on the local quotients $T_i$ of a Haefliger cocycle $\{(U_i,\pi_i,\gamma_{ij})\}$ defining $\f$. The holonomy pseudogroup $\mathscr{H}_\f$ then becomes a pseudogroup of local isometries of $T_\f$ and, with respect to a a bundle-like metric, the submersions defining $\f$ become Riemannian submersions. 

A Riemannian foliation $(M,\f)$ is \textit{complete} when $M$ is a complete Riemannian manifold with respect to some bundle-like metric for $\f$. The basic cohomology of a complete Riemannian foliation is finite dimensional, provided $M/\overline{\f}$ is compact (see, for instance, \cite[Proposition 3.11]{goertsches}). In particular $\chi(\f)$ is always defined in this case.




It follows from Molino's structural theorems (see \cite[Theorems 5.1 and 5.2]{molino}) that if $(M,\f)$ is a complete Riemannian foliation, then the partition $\overline{\f}$ is a singular Riemannian foliation of $M$ and there is a locally constant sheaf of Lie algebras of transverse Killing vector fields $\mathscr{C}_{\f}$ whose orbits are the closures of the leaves of $\f$, in the sense that
$$\{X_x\ |\ X\in(\mathscr{C}_{\f})_x\}\oplus T_xL_x=T_x\overline{L_x}.$$
The typical stalk $\mathfrak{g}_{\f}$ of $\mathscr{C}_{\f}$ is an important algebraic invariant called the \textit{structural Lie algebra} of $\f$.

We will be mostly interested in complete Riemannian foliations that have a \textit{globally} constant Molino sheaf, the so called \textit{Killing foliations}. In other words, $\f$ is a Killing foliation when there exists $\overline{X}_1,\dots,\overline{X}_d\in\mathfrak{iso}(\f)$ such that $T\overline{\f}=T\f\oplus\langle X_1,\dots, X_d \rangle$. It follows from Molino's theory that in this case $\mathscr{C}_{\f}(M)$ is central in $\mathfrak{l}(\f)$, hence the structural algebra of a Killing foliation is Abelian. For this reason we will also denote $\mathfrak{a}=\mathfrak{g}_{\f}$ when $\f$ is Killing (and it is understood from the context that we are referring to the structural algebra of $\f$).

A complete Riemannian foliation $\f$ of a simply-connected manifold is a Killing foliation \cite[Proposition 5.5]{molino}, since in this case $\mathscr{C}_{\f}$ cannot have holonomy. Homogeneous Riemannian foliations provide another relevant class of Killing foliations.

\begin{example}\label{example: homogeneous foliations are killing}
If $\f$ is a Riemannian foliation of a complete manifold $M$ given by the connected components of the orbits of a locally free action of $H<\mathrm{Iso}(M)$, then $\f$ is a Killing foliation and $\mathscr{C}_{\f}(M)$ consists of the transverse Killing vector fields induced by the action of $\overline{H}\subset\mathrm{Iso}(M)$ (see \cite[Lemme III]{molino3}).
\end{example}

\section{Equivariant basic cohomology}

In this section we recall the equivariant basic cohomology of a Killing foliation introduced in \cite{goertsches} and prove some facts that will be needed later. Let us begin by recalling the language of $\mathfrak{g}^\star$-algebras, which provides a purely algebraic setting for equivariant cohomology. We refer to \cite{guillemin} for a thorough introduction of this topic.

\subsection{$\mathfrak{g}^\star$-algebras}

Let $(A, d)$ be a differential $\mathbb{Z}$-graded commutative algebra and $\mathfrak{g}$ a finite-dimensional Lie algebra. We say that $A$ is a \textit{$\mathfrak{g}^\star$-algebra} if, for each $X\in\mathfrak{g}$, there are derivations $\mathcal{L}_X:A\to A$ and $\iota_X:A\to A$, of degree $0$ and $-1$ respectively, such that
$$\iota_X^2=0,\ \ \ \ [\mathcal{L}_X,\mathcal{L}_Y]=\mathcal{L}_{[X,Y]},\ \ \ \ [\mathcal{L}_X,\iota_Y]=\iota_{[X,Y]},\ \ \ \ \mbox{and}\ \ \ \ \mathcal{L}_X=d \iota_X+\iota_X d.$$
If $A$ and $B$ are $\mathfrak{g}^\star$-algebras, an algebra morphism $f:A\to B$ is a morphism of $\mathfrak{g}^\star$-algebras if it commutes with $d$, $\mathcal{L}_X$ and $\iota_X$.

An infinitesimal action of $\mathfrak{g}$ on a pseudogroup of local diffeomorphisms $(T,\mathscr{H})$ is a Lie algebra homomorphism $\mu:\mathfrak{g}\to \mathfrak{X}(\mathscr{H}):=\mathfrak{X}(T)^{\mathscr{H}}$.

\begin{proposition}
An infinitesimal action $\mu:\mathfrak{g}\to \mathfrak{X}(\mathscr{H})$ induces a $\mathfrak{g}^\star$-algebra structure on $\Omega(\mathscr{H}):=\Omega(T)^{\mathscr{H}}$ with the usual operators $d$, $\mathcal{L}_X$ and $\iota_X$.
\end{proposition}

\begin{proof}
Let $\gamma\in\mathscr{H}$, $X\in\mathfrak{X}(\mathscr{H})$ and $\omega\in\Omega(\mathscr{H})$. From naturality of $d$ we have $\gamma^*(d\omega)=d(\gamma^*\omega)=d\omega$ and, denoting $\gamma^*(X)=\dif \gamma^{-1}\circ X\circ\gamma$, we have $\gamma^*(\iota_X\omega)=\iota_{\gamma^*X}(\gamma^*\omega)=\iota_X\omega$ and $\gamma^*(\mathcal{L}\omega)=\mathcal{L}_{\gamma^*X}(\gamma^*\omega)=\mathcal{L}_X\omega$.
\end{proof}

We will say that an equivalence $\Phi$ between $(T,\mathscr{H})$ and another pseudogroup $(S,\mathscr{K})$ with an infinitesimal $\mathfrak{g}$-action $\nu$ is \textit{$\mathfrak{g}$-equivariant} if $\varphi^*(\nu(X))=\mu(X)$ for all $X\in\mathfrak{g}$ and $\varphi\in\Phi$. In this case $\Phi^*:\Omega(\mathscr{K})\to\Omega(\mathscr{H})$ is an isomorphism of $\mathfrak{g}^\star$-algebras.

\begin{example}
An infinitesimal action $\mu:\mathfrak{g}\to \mathfrak{X}(\mathcal{O})$ of $\mathfrak{g}$ on an orbifold $\mathcal{O}$ induces a $\mathfrak{g}^\star$-algebra structure on $\Omega(\mathcal{O})$.
\end{example}

A $\mathfrak{g}^\star$-algebra $A$ is \textit{acyclic} if $H(E,d)=\mathbb{R}$. Also, $A$ is said to be \textit{locally free}, or to be \textit{of type (C)}, if it admits a \textit{connection}, that is, an invariant element $\theta\in A_1\otimes\mathfrak{g}$ satisfying $\iota_X\theta=X$ for all $X\in\mathfrak{g}$. Equivalently, $A$ is locally free when there are $\theta^i\in A_1$ such that $\iota_{X_j}\theta^i=\delta^i_j$, for some basis $\{X_i\}$ of $\mathfrak{g}$, and
$$C=\spannn\{\theta^1,\dots,\theta^{\dim\mathfrak{g}}\}$$
is $\mathfrak{g}$-invariant.

\begin{proposition}\label{prop: isometric action on orbifold is of type C}
If a connected, compact Lie group $G$ acts on an orbifold $\mathcal{O}$ and $H<G$ is a Lie subgroup whose action is locally free, then $\Omega(\mathcal{O})$ is locally free as an $\mathfrak{h}^\star$-algebra.
\end{proposition}

\begin{proof}
Since $G$ is compact we can choose a $G$-invariant Riemannian metric on $\mathcal{O}$, and since the $H$-action is locally free the fundamental vector fields of the induced infinitesimal $\mathfrak{h}$-action $\mu$ are nowhere vanishing. This gives us an $H$-invariant splitting $T\mathcal{O}= T\f_H\oplus (T\f_H)^\perp$, where $T\f_H=\spannn(\mu(\mathfrak{h}))$ is the subbundle given by the tangent spaces of the $H$-orbits (see also \cite[Lemma 2.11]{galazgarcia}). Picking a basis $\{X_i\}$ of $\mathfrak{h}$ and defining $\theta^i\in\Omega^1(\mathcal{O})$ by $\iota_{X_j}\theta^i=\delta^i_j$ and $\theta^i|_{\nu\f_H}=0$ we obtain the desired connection.
\end{proof}

The \textit{basic complex} of a $\mathfrak{g}^\star$-algebra $A$ is
$$A_{\bas\mathfrak{g}}:=\{\omega\in A\ |\ \iota_X\omega=0\ \mbox{and}\  \mathcal{L}_X\omega=0\ \mbox{for all}\ X\in\mathfrak{g}\}.$$
It is easy to check that $A_{\bas\mathfrak{g}}$ is $d$-invariant. If $A$ is locally free then
\begin{equation}\label{eq free then ec is bc}H_{\mathfrak{g}}(A)=H(A_{\bas\mathfrak{g}}),\end{equation}
which is a generalization of the fact that, for a free Lie group action, the equivariant cohomology is the cohomology of the quotient (see, for instance, \cite[Section 5.1]{guillemin}). This leads one to define the \textit{Weil model for the equivariant cohomology of $A$} as
$$H_{\mathfrak{g}}(A):=H((E\otimes A)_{\bas\mathfrak{g}},d),$$
where $E$ is any locally free acyclic $\mathfrak{g}^\star$-algebra (see \cite[Section 2.4]{guillemin} for details). The analogy with classical equivariant cohomology is that $E$ replaces the classifying $G$-space and basic cohomology replaces the cohomology of the quotient.

There is an alternative model for the equivariant cohomology of a $\mathfrak{g}^\star$-algebra $A$ due to H.~Cartan which will be useful. Consider the \textit{Cartan complex}
$$C_{\mathfrak{g}}(A):=(\sym(\mathfrak{g}^*)\otimes A)^{\mathfrak{g}},$$
where $\sym(\mathfrak{g}^*)$ is the symmetric algebra over $\mathfrak{g}^*$ and the superscript indicates the subspace of $\mathfrak{g}$-invariant elements, that is, those $\omega\in\sym(\mathfrak{g}^*)\otimes A$ such that $\mathcal{L}_X\omega=0$ for all $X\in\mathfrak{g}$. Notice that we can identify an element $\omega\in C_{\mathfrak{g}}(A)$ with a $\mathfrak{g}$-equivariant polynomial map $\omega:\mathfrak{g}\to A$ (see \cite[p. 53]{guillemin}). The \textit{equivariant differential} $d_\mathfrak{g}$ of the Cartan complex is then defined as
$$(d_\mathfrak{g}\omega)(X)=d(\omega(X))-\iota_X(\omega(X)).$$
It is a derivation of degree $1$ with respect to the usual grading $C_{\mathfrak{g}}^n(A)=\bigoplus_{2k+l=n}(\sym_k(\mathfrak{g}^*)\otimes A_l)^{\mathfrak{g}}$.

The \textit{Cartan model for the equivariant cohomology of $A$} is
$$H_{\mathfrak{g}}(A):=H(C_{\mathfrak{g}}(A),d_\mathfrak{g}).$$
There is a natural structure of $\sym(\mathfrak{g}^*)^{\mathfrak{g}}$-algebra on $H_{\mathfrak{g}}(A)$ induced by $\sym(\mathfrak{g}^*)^{\mathfrak{g}}\ni f\mapsto f\otimes 1\in C_{\mathfrak{g}}(A)$. When $A=\Omega(\mathcal{O})$, for an orbifold $\mathcal{O}$, this structure coincides with the one induced by the map $\mathcal{O}\to\{\mathrm{point}\}$ by pullback. By abuse of notation we denote both the Weil and the Cartan model by $H_{\mathfrak{g}}(A)$, because they are isomorphic via the Mathai--Quillen--Kalkman isomorphism (see \cite[Theorem 4.2.1]{guillemin}).

We end this section with an equivariant version of \eqref{eq free then ec is bc} known as the ``commuting actions principle'', which will be a useful tool.

\begin{proposition}[{\cite[Proposition 3.9]{goertsches}}, {\cite[Proposition A.6.3]{lin}}]\label{prop: commuting actions principle}
Let $\mathfrak{g}=\mathfrak{h}\times\mathfrak{k}$ be a product of Lie algebras and let $A$ be an $(\mathfrak{h}\times\mathfrak{k})^\star$-algebra which is locally free as an $\mathfrak{h}^\star$-algebra, admitting a $\mathfrak{k}$-invariant connection. Then the inclusion $j:(\sym(\mathfrak{k}^*)\otimes A_{\bas\mathfrak{h}})^{\mathfrak{k}}\to (\sym(\mathfrak{g}^*)\otimes A)^{\mathfrak{g}}$ induces an $\sym(\mathfrak{k}^*)^{\mathfrak{k}}$-algebra isomorphism
$$H_{\mathfrak{k}}(A_{\bas\mathfrak{h}})\cong H_{\mathfrak{g}}(A).$$
\end{proposition}

\begin{remark}
Proposition \ref{prop: commuting actions principle} is a generalization for Lie algebra actions of a classical theorem by H.~Cartan regarding equivariant cohomology of principal bundles with a commuting Lie group action (see, e.g., \cite{nicolaescu}). It appears in \cite[Proposition 3.9]{goertsches} with the additional hypothesis that $A_k=0$ for $k<0$, and was recently generalized in \cite[Proposition A.6.3]{lin} where the authors also prove that the inverse isomorphism is given by the so-called \textit{Cartan map} $\mathcal{C}_\theta$. In the Cartan model $\mathcal{C}_\theta$ is the map induced in cohomology by $C_{\mathfrak{g}}(A)\ni\omega\mapsto \hor_\theta(\omega(\mu_\mathfrak{k}))\in C_{\mathfrak{k}}(A_{\bas\mathfrak{h}})$, where $\omega(\mu_\mathfrak{k})$ denotes the polynomial $\mathfrak{k}\to A$ obtained from $\omega$ by substituting the $\mathfrak{k}$-equivariant curvature
$$\mu_\mathfrak{k}:=d_{\mathfrak{k}}\theta+\frac{1}{2}[\theta,\theta]\in C_{\mathfrak{k}}^2(A)\otimes\mathfrak{h}$$
for the $\mathfrak{h}$-variable, and $\hor_\theta$ denotes the $\theta$-horizontal projection. The restriction of $\mathcal{C}_\theta$ to $\sym(\mathfrak{h}^*)^\mathfrak{h}$ is the equivariant Chern--Weil homomorphism.
\end{remark}

\subsection{Equivariant cohomology of orbifolds}

Suppose a Lie group $G$ acts on $\mathcal{O}$. One can then form the Borel construction $\mathcal{O}_G:=EG\times_G|\mathcal{O}|$ and define the $G$-equivariant cohomology of $\mathcal{O}$ as $H_G(\mathcal{O}):=H(\mathcal{O}_G,\mathbb{R})$, where the latter is the singular cohomology of $\mathcal{O}_G$ with coefficients in $\mathbb{R}$.

On the other hand, there is an induced infinitesimal action of the Lie algebra $\mathfrak{g}$ of $G$ on $\mathcal{O}$, and we can consider the $\mathfrak{g}$-equivariant cohomology
$$H_\mathfrak{g}(\mathcal{O}):=H_\mathfrak{g}(\Omega(\mathcal{O})).$$
We have the following orbifold version of the equivariant De Rham theorem.

\begin{theorem}[Equivariant De Rham theorem for orbifolds]\label{thrm: equivariant De Rham theorem for orbifolds}
Let a connected, compact Lie group $G$ with Lie algebra $\mathfrak{g}$ act on an $n$-dimensional orbifold $\mathcal{O}$. Then
$$H_G(\mathcal{O})\cong H_\mathfrak{g}(\mathcal{O})$$
as $\sym(\mathfrak{g}^*)^{\mathfrak{g}}$-algebras.
\end{theorem}

\begin{proof}
We can reduce the proof to an application of the classical equivariant De Rham theorem (see, for instance, \cite[Theorem 2.5.1]{guillemin}) by recalling that $\mathcal{O}$ is diffeomorphic to $\mathcal{O}^{\Yup}_{\mathbb{C}}/\!/\mathrm{U}(n)$ (see Example \ref{example: orbifold structure on the leaf space}). The $G$-action on $\mathcal{O}$ lifts to an action on $\mathcal{O}^{\Yup}_{\mathbb{C}}$ commuting with the natural $\mathrm{U}(n)$-action. We then have
\begin{eqnarray*}
H_G(\mathcal{O}) & \cong &  H_G(\mathcal{O}^{\Yup}_{\mathbb{C}}/\mathrm{U}(n))\cong H_{G\times\mathrm{U}(n)}(\mathcal{O}^{\Yup}_{\mathbb{C}})\\
 & \cong & H_{\mathfrak{g}\times\mathfrak{u}(n)}(\mathcal{O}^{\Yup}_{\mathbb{C}})\cong H_{\mathfrak{g}}(\Omega_{\bas\mathfrak{u}(n)}(\mathcal{O}^{\Yup}_{\mathbb{C}}))\cong H_\mathfrak{g}(\Omega(\mathcal{O}))\\
 & = & H_{\mathfrak{g}}(\mathcal{O})
\end{eqnarray*}
where the third isomorphism follows from the classical equivariant De Rham theorem, since $\mathcal{O}^{\Yup}_{\mathbb{C}}$ is a manifold, and the fourth one from Proposition \ref{prop: commuting actions principle}.\end{proof}

\subsection{Equivariant basic cohomology}

A transverse infinitesimal action of $\mathfrak{g}$ on a foliated orbifold $(\mathcal{O},\f)$ is a Lie algebra homomorphism
$$\mu:\mathfrak{g}\longrightarrow \mathfrak{l}(\f)$$
A transverse infinitesimal action induces a $\mathfrak{g}^\star$-algebra structure on $\Omega(\f)$, with $d$ being the usual exterior derivative and the derivations $\mathcal{L}_X$ and $\iota_X$ defined as $\mathcal{L}_X\omega:=\mathcal{L}_{\widetilde{X}}\omega$ and $\iota_X\omega:=\iota_{\widetilde{X}}\omega$ (the proof in \cite[Proposition 3.12]{goertsches} for foliations on manifolds adapts directly).

Each $\overline{X}\in\mathfrak{l}(\f)$ corresponds to a unique $X_T\in\mathfrak{X}(\mathscr{H}_\f)$ that restricted to $T_i$ is $\pi_i$-related to the restriction $X|_{\widetilde{U}_i}$ of a foliate representative of $\overline{X}$. This correspondence induces an infinitesimal $\mathfrak{g}$-action on $(T_\f,\mathscr{H}_\f)$. Let $i:T_\f\to U_\mathcal{O}$ be the inclusion of $T_\f$ as a total transversal. We claim that $i^*:\Omega(\mathscr{H})\to\Omega(\f)$ is an isomorphism of $\mathfrak{g}^\star$-algebras. In fact, it is clear that it is an isomorphism of algebras and that $di^*=i^*d$. Moreover, for $\mu(X)_T\in\mathfrak{X}(\mathscr{H}_\f)$ one can always choose a representative $\widetilde{X}\in\mathfrak{L}(\f)$ of $\mu(X)$ that is $i$-related to $\mu(X)_T$. Then
$$i^*(\mathcal{L}_X\omega)=i^*(\mathcal{L}_{\widetilde{X}}\omega)=\mathcal{L}_{\mu(X)_T}i^*\omega=\mathcal{L}_Xi^*\omega$$
and similarly $i^*\iota_X=\iota_X i^*$, for each $X\in\mathfrak{g}$. This establishes the following.

\begin{proposition}\label{proposition: transverse action projects to holonomy pseudogroup}
A transverse action $\mathfrak{g}\to\mathfrak{l}(\f)$ projects to an infinitesimal action of $\mathfrak{g}$ on $(T_\f,\mathscr{H}_\f)$ and
$$\Omega(\f)\cong\Omega(\mathscr{H}_\f)$$
as $\mathfrak{g}^\star$-algebras.
\end{proposition}

The \textit{$\mathfrak{g}$-equivariant basic cohomology of $(\mathcal{O},\f)$} is the $\mathfrak{g}$-equivariant cohomology of the basic subcomplex $\Omega(\f)$, which we will denote by
$$H_\mathfrak{g}(\f):=H_\mathfrak{g}(\Omega(\f))=H(C_\mathfrak{g}(\Omega(\f),d_\mathfrak{g})).$$
By Proposition \ref{proposition: transverse action projects to holonomy pseudogroup} we therefore have $H_\mathfrak{g}(\f)\cong H_\mathfrak{g}(\Omega(\mathscr{H}_\f))$, as $\sym(\mathfrak{g}^*)^{\mathfrak{g}}$-algebras.

\begin{example}
A Killing foliation $(M,\f)$ has a natural transverse action of its structural algebra $\mathfrak{a}$, since $\mathfrak{a}\cong\mathscr{C}_\f(M)$. Notice that the fixed point set $M^\mathfrak{a}=\{x\in M\ |\ \mathfrak{a}_x=\mathfrak{a}\}$ is precisely the union of the closed leaves of $\f$, since $\mathfrak{a}\f=\overline{\f}$ by Molino's structural theory. We will be mainly interested in the study of $H_\mathfrak{a}(\f)$.
\end{example}

\section{Deformations of Killing foliations}\label{section: deformations}

Two $p$-dimensional smooth foliations $\f_0$ and $\f_1$ of $M$ are $C^\infty$-\textit{homotopic} if there is a $p$-dimensional smooth foliation $\f$ of $M\times [0,1]$ such that $M\times\{t\}$ is saturated by leaves of $\f$, for each $t\in[0,1]$, and
$$\f_i=\f|_{M\times\{i\}},$$
for $i=0,1$. Here we will simply say that $\f_t$ is a \textit{deformation} of $\f_0$ into $\f_1$.

It was shown by A.~Haefliger and E.~Salem \cite{haefliger2}, as a corollary of their study on the classifying space of holonomy pseudogroups of Killing foliations, that it is possible to deform such a foliation $\f$ into a closed foliation $\g$, that can be chosen arbitrarily close to $\f$, in a way that the deformation occurs within the closures of the leaves of $\f$. In \cite{caramello} we showed that some of the transverse geometry and topology of $\f$ is preserved throughout this deformation, so that one can use orbifold theory on $M/\!/\g$ to study $\f$. For the reader's convenience we summarize some properties of this deformation in the following.

\begin{theorem}[{\cite[Theorem B]{caramello}}]\label{theorem: deformation}
Let $(\f,\mathrm{g}^T)$ be a Killing foliation of a compact manifold $M$. Then there is a deformation $\f_t$ of $\f$, called a \emph{regular deformation}, into a closed foliation $\g$, which can be chosen arbitrarily close to $\f$, such that
\begin{enumerate}[(i)]
\item \label{injection of tensor algebra} for each $t$ there is an injection $\iota:\mathcal{T}(\f)\to\mathcal{T}(\f_t)$ that smoothly deforms transverse geometric structures given by $\f$-basic tensors, such as the metric $\mathrm{g}^T$, into respective transverse geometric structures for $\f_t$,
\item \label{toric action on the quotient} the quotient orbifold $M/\!/\g$ admits an effective isometric action of a torus $\mathbb{T}^d$, with respect to the metric induced from $\iota\mathrm{g}^T$, such that $M/\overline{\f}\cong(M/\g)/\mathbb{T}^d$, where $d=\dim\mathfrak{a}$.
\item \label{symmetries and basic tensors} $\mathcal{T}(\f)$ is isomorphic to the algebra $\mathcal{T}(M/\!/\g)^{\mathbb{T}^d}$ of $\mathbb{T}^d$-invariant tensor fields on $M/\!/\g$, the isomorphism being given by $\pi_*\circ\iota$, where $\pi_*:\mathcal{T}(\g)\to\mathcal{T}(M/\!/\g)$ is the pushforward by the canonical projection.
\end{enumerate}
In particular, if $\g$ is chosen sufficiently close to $\f$, upper and lower bounds on transverse sectional and Ricci curvature of $\f$ are maintained.
\end{theorem}

\begin{remark} In \cite{caramello} we mention, without using it, that an $\f$-basic symplectic form $\omega$ is mapped by $\iota$ into a $\g$-basic symplectic form, but this is not true in general because $\iota$ may not commute with the exterior derivative, so $\iota\omega$ may not be closed after the deformation.
\end{remark}

It will be useful to briefly recall how the regular deformation is constructed. By \cite[Theorem 3.4]{haefliger2}, there exists an orbifold $\mathcal{O}$ associated to $\f$ admitting a $\mathbb{T}^N$-action $\mu$ such that there is a dense contractible subgroup $H<\mathbb{T}^N$ that acts locally freely on $\mathcal{O}$ and the holonomy pseudogroup $\mathscr{H}_{\f_H}$ of the foliation $\f_H$ of $\mathcal{O}$ defined by the orbits of $H$ is equivalent to $\mathscr{H}_\f$. Moreover, there is a smooth (good) map $\Upsilon:M\to\mathcal{O}$ transverse to $\f_H$ with $\Upsilon^*(\f_H)=\f$. As the authors comment in \cite{haefliger2}, this construction can be used to deform $\f$ as follows. Let $\mathfrak{h}$ be the Lie algebra of $H$ and consider a Lie subalgebra $\mathfrak{k}<\mathrm{Lie}(\mathbb{T}^N)\cong\mathbb{R}^N$, with $\dim(\mathfrak{k})=\dim(\mathfrak{h})$, such that its corresponding Lie subgroup $K<\mathbb{T}^N$ is closed. Suppose $\mathfrak{k}$ is close enough to $\mathfrak{h}$, as points in the Grassmannian $\mathrm{Gr}^{\dim\mathfrak{h}}(\mathrm{Lie}(\mathbb{T}^N))$, so that one can choose a smooth path $\mathfrak{h}(t)$ connecting $\mathfrak{h}$ to $\mathfrak{k}$ such that for each $t$ the action $\mu|_{H(t)}$ of the corresponding Lie subgroup $H(t)$ is locally free and the induced foliation remains transverse to $\Upsilon$. Then $\f_t:=\Upsilon^*(\f_{H(t)})$ defines a $C^{\infty}$-homotopic deformation of $\f=\f_0$ into $\g=\f_1$. In this case $\mathscr{H}_{\f_t}$ is equivalent to $\mathscr{H}_{\f_{H(t)}}$ for each $t$. Moreover, since $K$ is closed, $\g$ is a closed foliation.

The map $\iota:\mathcal{T}(\f)\to\mathcal{T}(\f_t)$ is defined by working on $\mathcal{O}$ and conjugating with $\Upsilon^*$, which is an isomorphism $\mathcal{T}(\f_{H(t)})\to\mathcal{T}(\f_t)$. On $\mathcal{O}$ the map simply ``deforms the kernel'' of a tensor $\xi\in\mathcal{T}(\f_H)$, yielding a tensor $\xi_t\in\mathcal{T}(\f_{H(t)})$. More precisely, by taking $\mathfrak{k}$ closer to $\mathfrak{h}$ if necessary, we can fix a subalgebra $\mathfrak{a}<\mathfrak{t}$ which is complementary to each $\mathfrak{h}(t)$. By density we can furthermore assume that the corresponding $1$-parameter subgroup $A<\mathbb{T}^N$ is closed. Here we are purposefully abusing the notation by denoting this subalgebra by $\mathfrak{a}$ since its induced $\f_H$-transverse action will coincide, via the correspondence $\f \leftrightarrow \f_H$, with the natural $\f$-transverse action of the structural algebra of $\f$ (see also Section \ref{section: beq under deformations}). Now choose a $\mathbb{T}^N$-invariant Riemannian metric on $\mathcal{O}$ and consider the distribution of varying rank $D^\perp$ given by $D^\perp_x=T_x(\mathbb{T}^Nx)^\perp$. The distribution $D_\mathfrak{a}$ spanned by the induced infinitesimal $\mathfrak{a}$-action on $\mathcal{O}$ then complements $D^\perp$ to a $\mathbb{T}^N$-invariant subbundle $D=D_\mathfrak{a}\oplus D^\perp\subset T\mathcal{O}$ which is complementary to $T\f_{H(t)}$, for each $t$. We declare that $\xi_t$ coincides with $\xi$ on $D$ and vanishes when contracted with vectors in $T\f_{H(t)}$. In particular, one sees that $\iota\mathrm{g}^T$ is a transverse metric for $\f_t$. The $\mathbb{T}^N$-action on $\mathcal{O}$ projects to an action of $\mathbb{T}^N/K\cong\mathbb{T}^d$ on $\mathcal{O}/\!/K$, which is diffeomorphic to $M/\!/\g$, thus yielding us the $\pi_*\circ\iota(\mathrm{g}^T)$-isometric $\mathbb{T}^d$-action on $M/\!/\g$.

\subsection{Applications of the deformation technique}

We now illustrate the application of this deformation method by establishing Theorems \ref{theorem: bochner intro}, \ref{theo: Milnor trasnverso intro} and \ref{theorem: synge intro} from the introduction. Theorem \ref{theorem: bochner intro} consists of the following theorem and its corollary.

\begin{theorem}\label{theorem: transverse bochner}
Let $\f$ be a Killing foliation of a compact manifold $M$. If $\ric_\f\leq c<0$, then $\f$ is closed.
\end{theorem}

\begin{proof}
Suppose $\f$ is not closed, so $\dim(\mathfrak{a})>0$. Passing to a double cover of $M$ if necessary, we can suppose $\f$ is transversely orientable \cite[Proposition 3.5.1]{candel}. In fact, the pullback of any non-closed Riemannian foliation by a finitely-sheeted covering map is also non-closed (see \cite[Proposition 3.6]{caramello}). Then we can apply Theorem \ref{theorem: deformation} to obtain a closed Riemannian foliation $\g$ also satisfying $\ric_{\g}<0$, hence $M/\!/\g$ is a connected, compact, oriented Riemannian orbifold  with $\ric_{\mathcal{O}}<0$. Moreover, item \eqref{toric action on the quotient} of Theorem \ref{theorem: deformation} asserts that $M/\!/\g$ admits an isometric action of a torus $\mathbb{T}^d$, with $d=\dim(\mathfrak{a})>0$. This contradicts Theorem \ref{teo: bochner for orbifolds}.
\end{proof}

\begin{corollary}
Let $\f$ be a complete, transversely compact Riemannian foliation of a manifold $M$ with $|\pi_1(M)|<\infty$. If $\ric_\f\leq c<0$, then $\f$ is closed.
\end{corollary}

\begin{proof} Suppose $\f$ is not closed and consider its lift $\widetilde{\f}$ to the universal cover $\widetilde{M}$ of $M$. Since $\widetilde{M}$ is simply connected and $|\pi_1(M)|<\infty$, we have that $\widetilde{\f}$ is a non-closed Killing foliation. We claim that $\widetilde{M}/\overline{\widetilde{\f}}$, endowed with the length structure induced by the transverse metric, is bounded. In fact, the distance between two leaf closures in $\overline{\widetilde{\f}}$ will be at most $|\pi_1(M)|$ times the distance between their projections on $M$, which is bounded by $\diam(M/\overline{\f})<\infty$. It follows that $\widetilde{\f}$ is transversely compact by the Hopf--Rinow theorem for length spaces (see, for instance, \cite{gromov}).

From the surjection \eqref{surjection fundamental groups} we obtain that $\mathscr{H}_{\widetilde{\f}}$ is a simply connected, complete pseudogroup of isometries whose orbit space $\widetilde{T}/\mathscr{H}_{\widetilde{\f}}\cong \widetilde{M}/\widetilde{\f}$ is compact. By \cite[Theorem 3.7]{haefliger2}, there exists a Killing foliation $\f'$ of a simply connected, \emph{compact} manifold $M'$, such that $\mathscr{H}_{\widetilde{\f}}\cong\mathscr{H}_{\f'}$. Since the pseudougroups are equivalent $\f'$ is also not closed and admits a transverse Riemannian metric satisfying $\ric_{\f'}\leq c<0$. This contradicts Theorem \ref{theorem: transverse bochner}.
\end{proof}

The next result is a transverse generalization of Synge's theorem. It restricts to the classical Synge's theorem for manifolds when $\f$ is the trivial foliation by points.

\begin{theorem}\label{theorem: synge for foliations}
Let $\f$ be a Killing foliation of a compact manifold $M$, with $\sec_\f >0$. Then
\begin{enumerate}[(i)]
\item if $\codim\f$ is even and $\f$ is transversely orientable, then $M/\overline{\f}$ is simply connected, and
\item if $\codim\f$ is odd and $\mathrm{Hol}(L)$ preserves transverse orientation for each $L\in \f$, then $\f$ is transversely orientable.
\end{enumerate}
\end{theorem}

\begin{proof}
By deforming $\f$ via Theorem \ref{theorem: deformation} we obtain a closed Riemannian foliation $\g$ of $M$ with $\sec_\g >0$ and $\codim\g=\codim\f$. Recall that $\f=\Upsilon^{-1}(\f_H)$ and $\g=\Upsilon^{-1}(\f_K)$, for subgroups $H,K<\mathbb{T}^N$ acting on the Haefliger--Salem orbifold $\mathcal{O}$.

Suppose $\codim\f$ is even and $\f$ is transversely orientable. Then map $\iota$ sends a transverse volume form for $\f$ into a transverse volume form for $\g$, so $\g$ is transversely orientable and $M/\!/\g$ is an orientable, compact, positively curved Riemannian orbifold. It follows that $|M/\!/\g|$ is simply connected, by the orbifold version of Synge's theorem that appears in \cite[Corollary 2.3.6]{yeroshkin}. By Theorem item \eqref{toric action on the quotient} of \ref{theorem: deformation} we have
$$\frac{|M/\!/\g|}{\mathbb{T}^d}\cong \frac{M}{\overline{\f}},$$
therefore $M/\overline{\f}$ simply connected, since it is the quotient of a simply connected space by the action of a connected, compact Lie group (see \cite[Example 4]{armstrong}).

Now suppose $\codim\f$ is odd and $\mathrm{Hol}(L)$ consists of germs of maps that preserve transverse orientation, for each $L\in \f$. Let $Hx$ be the orbit on $\mathcal{O}$ corresponding to $L$. For an $H_x$-invariant chart $(\widetilde{U},\Gamma_x,\phi)$ around $x$ there is an extension $\widetilde{H}_x$ of $H_x$ by $\Gamma_x$ that acts on $\widetilde{U}$ with $\widetilde{U}/\widetilde{H}_x\cong U/H_x$ \cite[Proposition 2.12]{galazgarcia}. The group $\widetilde{H}_x$ is the holonomy group of the leaf $Hx\in\f_H$, hence, through the equivalence $\mathscr{H}(\f)=\mathscr{H}(\f_H)$, the hypothesis on $\mathrm{Hol}(L)$ implies that $\widetilde{H}_x$ (and in particular $\Gamma_x$) preserves transverse orientation. Passing to the subgroup $K<\mathbb{T}^N$ that defines $\g$, we thus obtain that $\widetilde{K}_x$, an extension of $K_x<\mathbb{T}^N$ by $\Gamma_x$, also preserves transverse orientation. Hence $\mathcal{O}/\!/\f_K\cong M/\!/\g$ is an odd-dimensional, compact, locally orientable, positively curved orbifold, so again it follows from \cite[Corollary 2.3.6]{yeroshkin} that $M/\!/\g$ is orientable (thus $\g$ is transversely orientable). The Riemannian volume form of $M/\!/\g$ is $\mathbb{T}^d$-invariant, hence induces a transverse volume form for $\f$, by item \eqref{symmetries and basic tensors} of Theorem \ref{theorem: deformation}.
\end{proof}

Finally, let us prove the transverse version of Milnor's theorem on the growth of the fundamental group.

\begin{theorem}\label{theorem: negative curvature exponential growth}
Let $\f$ be a Killing foliation on a compact manifold $M$ such that $\sec_\f<0$. Then $\f$ is closed and $\pi_1(\f)$ grows exponentially. In particular, $\pi_1(M)$ grows exponentially.
\end{theorem}

\begin{proof}
If $\f$ were not closed then, by Theorem \ref{theorem: deformation}, we would be able to deform $\f$ into a closed foliation $\g$ with $\sec_\g<0$, and $M/\!/\g$ would be a negatively curved orbifold admitting an isometric action of $\mathbb{T}^d$, $d>0$, contradicting Theorem \ref{teo: bochner for orbifolds}. So $\f$ is closed and hence
$$\pi_1(\f)=\piorb(M/\!/\f)$$
has exponential growth, by Theorem \ref{theorem: Milnor growth for orbifolds}. In particular, from the surjection \eqref{surjection fundamental groups} it follows that $\pi_1(M)$ also grows exponentially.
\end{proof}

Gromov establishes in \cite[p. 12]{gromov3} an upper bound for the total sum of Betti numbers for negatively curved manifolds in terms of their dimension and volume. This result was generalized for orbifolds by I.~Samet in \cite[Theorem 1.1]{samet} (recall that every orbifold with negative sectional curvature is a quotient of a Hadamard manifold by a discrete group of isometries, see \cite[Corollary 2.16]{bridson}). Combining it with Theorem \ref{theorem: negative curvature exponential growth} we get:

\begin{corollary}
There exists a constant $C=C(q)$ such that, for any Killing foliation $\f$ on a compact manifold $M$ with $\sec_\f<0$, say $-k^2\leq \sec_\f <0$, one has
$$\sum_{i=1}^q b_i(\f)\leq Ck^q\vol(M/\!/\f).$$
\end{corollary}

\section{Equivariant basic cohomology under deformations}\label{section: beq under deformations}

We have proved in \cite[Theorem 7.4]{caramello} that the basic Euler characteristic $\chi(\f_t)$ is constant in $t$ for a regular deformation $\f_t$. Despite this fact, the basic Betti numbers are not preserved, as the following example shows.

\begin{example}\label{exe: nozawa}
We are grateful to H.~Nozawa for communicating this example to us. Consider $M=\mathbb{S}^3\times\mathbb{S}^1$ with the free action of $\mathbb{T}^2=\mathbb{S}^1\times\mathbb{S}^1$ given by
$$((s_1,s_2),((z_1,z_2),z))\longmapsto ((s_1z_1,s_1z_2), s_2z).$$
Let $\f$ be the homogeneous Riemannian foliation of $M$ induced by a dense $1$-parameter subgroup of $\mathbb{T}^2$. In this case the construction of the orbifold $(\mathcal{O}_\f,\mathbb{T}^N)$ is trivial, that is, $\mathcal{O}_\f=M$ and $H$ is the $1$-parameter subgroup defining $\f$. It is clear that $\f$ can be regularly deformed to both the foliations $\g_1$ and $\g_2$, induced by the actions of $\mathbb{S}^1\times \{1\}$ and $\{1\}\times \mathbb{S}^1$, respectively. But $H(\g_1)=H(M/\!/\g^1)=H(\mathbb{S}^2\times \mathbb{S}^1)$, while $H(\g_2)=H(M/\!/\g_2)=H(\mathbb{S}^3)$. Hence $b_i(\g_1)\neq b_i(\g_2)$ for $i=1,2$.
\end{example}

Our next goal in this section is to show that, in spite of this, the equivariant basic cohomology is preserved throughout a regular deformation $\f_t$. To make this more precise we need to clarify how $\mathfrak{a}$ acts transversely on $\g$. We retain the notation of Section \ref{section: deformations}. For a Killing foliation $\f$ on a compact manifold $M$, consider the Haefliger--Salem construction $(\mathcal{O},H<\mathbb{T}^N)$ and define $\f_t:=\Upsilon^*(\f_{H(t)})$. Since the holonomy pseudogroups of $\f_t$ and $\f_{H(t)}$ are equivalent, there is an identification between $\f_t$-transverse vector fields on $M$ and $\f_{H(t)}$-transverse vector fields on $\mathcal{O}$. So we only have to show that $\mathfrak{a}$ acts transversely on $\f_{H(t)}$.

The $\mathfrak{a}$-action on $\mathscr{H}(\f)$ is precisely the action of the structural algebra of $\mathscr{H}(\f)$ as a complete pseudogroup of local isometries, which is characterized by the vector fields on $T_\f$ whose local flows belong to the $C^1$-closure $\overline{\mathscr{H}(\f)}$ (see \cite[\S 2.5]{haefliger2} for more details). Via the equivalence $\mathscr{H}(\f)\cong\mathscr{H}(\f_H)$ it coincides with the action of the structural algebra of $\mathscr{H}(\f_H)$, which in turn is the algebra of vector fields on $T_\f$ induced by the fundamental vector fields of the action of $\mathbb{T}^N$ on $\mathcal{O}$. In other words, the transverse $\mathfrak{a}$-action on $\f$ corresponds to the natural transverse $\mathfrak{t}/\mathfrak{h}$-action on $\f_H$, where $\mathfrak{t}$ is the Lie algebra of $\mathbb{T}^N$, similarly to the case of a homogeneous foliation on a manifold (see Example \ref{example: homogeneous foliations are killing}).

We indeed have a natural transverse action of $\mathfrak{t}/\mathfrak{h}(t)$ on $\f_{H(t)}$ for each $t$. Since all those Lie algebras are (non-canonically) isomorphic to $\mathfrak{a}$, we define an $\mathfrak{a}$-action on $\f_{H(t)}$, and hence on $\f_t$, by passing through an isomorphism $\nu_t:\mathfrak{t}/\mathfrak{h}(t)\to\mathfrak{a}$. This amounts to identifying $\mathfrak{a}$ with the subalgebra of $\mathfrak{t}$ complementary to each $\mathfrak{h}(t)$ that we used when constructing regular deformations (see Section \ref{section: deformations}), which by abuse of notation we also denoted by $\mathfrak{a}<\mathfrak{t}$. That is, $\nu_t$ is the map that sends $[X]\in\mathfrak{t}/\mathfrak{h}(t)$ to its unique representative in $\mathfrak{a}<\mathfrak{t}$.

In particular, the $\mathbb{T}^d$-action on $M/\!/\g$ appearing in item \eqref{toric action on the quotient} of Theorem \ref{theorem: deformation} is given by the natural action of $\mathbb{T}^N/K$ on $\mathcal{O}/\!/K\cong M/\!/\g$, so the transverse $\mathfrak{a}$-action on $\g$ is the lift of the induced infinitesimal action of $\mathrm{Lie}(\mathbb{T}^d)=\mathfrak{t}/\mathfrak{k}\cong\mathfrak{a}$ on $ M/\!/\g$. We sum this up in the following.

\begin{proposition}\label{prop: induced transverse a action on G}
The structural algebra $\mathfrak{a}$ of $\f$ acts transversely on $\f_t$, for each $t$, and its induced action on the quotient orbifold $M/\!/\g$, for the closed foliation $\g=\f_1$, integrates to the $\mathbb{T}^d$-action given by item (ii) of Theorem \ref{theorem: deformation}.
\end{proposition}

Now that Proposition \ref{prop: induced transverse a action on G} allows us to consider $H_\mathfrak{a}(\f_t)$ we can state the following.

\begin{theorem}\label{thrm: invariance of H_a under deformations}
Let $\f$ be a Killing foliation of a compact manifold $M$ and let $\f_t$ be a regular deformation of $\f$. Then for each $t$ there is an $\mathbb{R}$-algebra isomorphism
$$H_{\mathfrak{a}}(\f)\cong H_{\mathfrak{a}}(\f_t).$$
\end{theorem}

\begin{proof}
Using Proposition \ref{proposition: transverse action projects to holonomy pseudogroup} and the equivalence $\mathscr{H}_{\f_t}\cong\mathscr{H}_{\f_{H(t)}}$ we have
$$H_{\mathfrak{a}}(\f_t)\cong H_{\mathfrak{a}}(\mathscr{H}_{\f_t})\cong H_{\mathfrak{a}}(\mathscr{H}_{\f_{H(t)}})\cong H_{\mathfrak{a}}(\f_{H(t)})$$
as $\sym(\mathfrak{a}^*)$-algebras. Notice moreover that $H_{\mathfrak{a}}(\f_{H(t)})= H_{\mathfrak{a}}(\Omega_{\bas\mathfrak{h}(t)}(\mathcal{O}))$, since $\Omega_{\bas\mathfrak{h}(t)}(\mathcal{O})=\Omega(\f_{H(t)})$. We only need, thus, to exhibit an $\mathbb{R}$-algebra isomorphism $H_{\mathfrak{a}}(\Omega_{\bas\mathfrak{h}(0)}(\mathcal{O}))\cong H_{\mathfrak{a}}(\Omega_{\bas\mathfrak{h}(t)}(\mathcal{O}))$ for each $t$.

By Proposition \ref{prop: isometric action on orbifold is of type C}, let $\theta_t\in\Omega^1(\mathcal{O})\otimes\mathfrak{h}$ be a connection for the $H(t)$-action on $\mathcal{O}$. Since the infinitesimal action of $\mathfrak{a}$ on each $\f_t$ is induced by the action of a compact subgroup $A<\mathbb{T}^N$ we can suppose $\theta$ is $\mathfrak{a}$-invariant, by averaging it if necessary. From Proposition \ref{prop: commuting actions principle} it then follows that the inclusion $(\sym(\mathfrak{a}^*)\otimes \Omega_{\bas\mathfrak{h}(t)})^{\mathfrak{a}}\to (\sym(\mathfrak{t}^*)\otimes \Omega(\mathcal{O}))^{\mathfrak{t}}$ induces an $\sym(\mathfrak{a}^*)$-algebra isomorphism
$$H_{\mathfrak{a}}(\Omega_{\bas\mathfrak{h}(t)}(\mathcal{O})) \stackrel{j_t}{\longrightarrow} H_{\mathfrak{h}(t)\times\mathfrak{a}}(\mathcal{O}) = H_{\mathfrak{t}}(\mathcal{O}).$$
For each $t$, the map $j_t^{-1}\circ j_0: H_{\mathfrak{a}}(\Omega_{\bas\mathfrak{h}(0)}(\mathcal{O}))\to H_{\mathfrak{a}}(\Omega_{\bas\mathfrak{h}(t)}(\mathcal{O}))$ is clearly both an $\mathbb{R}$-linear and a ring isomorphism, hence it is the desired $\mathbb{R}$-algebra isomorphism.
\end{proof}

\begin{remark}
One does not have $H_{\mathfrak{a}}(\f)\cong H_{\mathfrak{a}}(\f_t)$ as $\sym(\mathfrak{a}^*)$-algebras, in general. In fact, each $j_t$ is an $\sym(\mathfrak{a}^*)$-algebra isomorphism, but with respect to the decomposition $\mathfrak{t}=\mathfrak{h}(t)\times\mathfrak{a}$. For each $t$ this decomposition induces a different $\sym(\mathfrak{a}^*)$-module structure on $H_{\mathfrak{t}}(\mathcal{O})$, since the extension of a polynomial on $\mathfrak{a}$ to a polynomial on $\mathfrak{t}$ vanishes on $\mathfrak{h}(t)$. Hence $j_t^{-1}\circ j_0$ fails to preserve the module structure.
\end{remark}

In particular, for $t=1$ it follows from Theorem \ref{thrm: invariance of H_a under deformations} that $H_{\mathfrak{a}}(\f)\cong H_{\mathfrak{a}}(\g)$, for a closed foliation $\g$ arbitrarily close to $\f$.

\begin{corollary}\label{cor: bec in terms of the approximated quotient orbifold}
Let $\f$ be a Killing foliation of a compact manifold $M$. Then there is a closed approximation $\g$ of $\f$ such that $M/\!/\g$ admits a $\mathbb{T}^d$-action, $d=\dim\mathfrak{a}$, and
$$H_\mathfrak{a}(\f)\cong H_{\mathbb{T}^d}(M/\!/\g),$$
as $\mathbb{R}$-algebras.
\end{corollary}

\begin{proof}
We have $H_{\mathfrak{a}}(\f) \cong H_{\mathfrak{a}}(\g)$ from Theorem \ref{thrm: invariance of H_a under deformations}, $H_{\mathfrak{a}}(\g)\cong H_{\mathfrak{a}}(M/\!/\g)$ from Proposition \ref{proposition: transverse action projects to holonomy pseudogroup} and $H_{\mathfrak{a}}(M/\!/\g)\cong H_{\mathbb{T}^d}(M/\!/\g)$ from Proposition \ref{prop: induced transverse a action on G} and Theorem \ref{thrm: equivariant De Rham theorem for orbifolds}.
\end{proof}

By deforming $\f$ we obtain, therefore, a ``topological model'' for $H_\mathfrak{a}(\f)$, as a ring, given by the cohomology $H_{\mathbb{T}^d}(M/\!/\g)$ of the orbifold $M/\!/\g$. The orbifold $(\mathcal{O},\mathbb{T}^N)$, on the other hand, acts as a topological model for the full $\sym(\mathfrak{a}^*)$-algebra structure of $H_\mathfrak{a}(\f)$. In fact, by Proposition \ref{prop: commuting actions principle} and Theorem \ref{thrm: equivariant De Rham theorem for orbifolds} it follows that
$$H_{\mathfrak{a}}(\f)\cong H_{\mathfrak{a}}(\f_H)\cong H_{\mathfrak{a}}(\Omega_{\bas\mathfrak{h}}(\mathcal{O}))\cong H_{\mathbb{T}^N}(\mathcal{O}).$$
as $S(\mathfrak{a}^*)$-algebras. We state this below.

\begin{corollary}
Let $\f$ be a Killing foliation of a compact manifold $M$ and let $(\mathcal{O},H<\mathbb{T}^N)$ be its Haefliger--Salem construction. Then
$$H_{\mathfrak{a}}(\f)\cong H_{\mathbb{T}^N}(\mathcal{O})$$
as $\sym(\mathfrak{a}^*)$-algebras.
\end{corollary}

\section{Formal actions and basic Betti numbers}

A $\mathfrak{g}^\star$-algebra $A$ is said to be \textit{equivariantly formal} if $\sym(\mathfrak{g}^*)^\mathfrak{g}\otimes H(A)$ as $\sym(\mathfrak{g}^*)^\mathfrak{g}$-modules. Equivalently, $A$ is equivariantly formal when $H_\mathfrak{a}(A)$ is a free $\sym(\mathfrak{a}^*)$-module. In the case of a Killing foliation $\f$, if $\Omega(\f)$ is equivariantly formal as an $\mathfrak{a}^\star$-algebra with respect to the natural transverse action of the structural algebra $\mathfrak{a}$, we say for short that $\f$ is equivariantly formal. In the following example we summarize some characterizations and sufficient conditions for equivariant formality of a Killing foliation $\f$ that appear in \cite[Sections 3, 5 and 8]{goertsches}.

\begin{example}\label{exe: charracterization formality}
Let $\f$ be a Killing foliation with structural algebra $\mathfrak{a}$. The following properties are equivalent:
\begin{enumerate}[(i)]
\item $\f$ is equivariantly formal.
\item $H_\mathfrak{a}(\f)$ is a free $\sym(\mathfrak{a}^*)$-module.
\item The natural map $H_\mathfrak{a}(\f)\to H(\f)$ is surjective.
\end{enumerate}
Moreover, any of the conditions below is sufficient for them to hold, provided $\f$ is transversely orientable:
\begin{enumerate}[(i)]
\setcounter{enumi}{3}
\item $H^{\mathrm{odd}}(\f)=0$.
\item $\dim H(M^{\mathfrak{a}}/\!/\f)=\dim H(\f)$.
\item $\f$ admits a basic Morse-Bott function whose critical set is equal to $M^{\mathfrak{a}}$,
\end{enumerate}
\end{example}

Although the definition of equivariant formality of $\f$ is given in terms of the $\sym(\mathfrak{a}^*)$-module structure of $H_{\mathfrak{a}}(\f)$, there is yet another way to study this property which only takes into account the ring structure of $H_{\mathfrak{a}}(\f)$. This will be of interest to us, since the module structure of $H_{\mathfrak{a}}(\f)$ is not preserved through regular deformations. Let $M$ be a finitely generated $R$-module, where $R$ is a non-trivial Noetherian commutative local ring with identity. Then one can define its Krull dimension, $\dim_RM$, and depth (see, for instance, \cite[Appendix A]{allday}), which always satisfy $\depth M\leq\dim_RM$. When equality holds $M$ is said to be a \textit{Cohen--Macaulay module}. Analogously, the ring $R$ is a \textit{Cohen--Macaulay ring} when it is a Cohen--Macaulay module over itself. We refer to \cite[Section A.6]{allday} and \cite[Section IV.B]{serre} for more details on this topic. For us it will be sufficient to know that if the transverse $\mathfrak{a}$-action on a Killing foliation $\f$ of a compact manifold is equivariantly formal then $H_{\mathfrak{a}}(\f)$ is a Cohen--Macaulay $\sym(\mathfrak{a}^*)$-module, which in turn happens if, and only if, $H_{\mathfrak{a}}(\f)$ is a Cohen--Macaulay ring \cite[Proposition B.3]{goertsches}.

\begin{proposition}\label{prop: formality is preserved}
Let $\f$ be a Killing foliation of a compact manifold $M$ and let $\f_t$ be a regular deformation. If the transverse $\mathfrak{a}$-action on $\f$ is equivariantly formal, then so is its induced action on each $\f_t$.
\end{proposition}

\begin{proof}
If $\f$ is equivariantly formal then $H_{\mathfrak{a}}(\f)$ is a Cohen--Macaulay ring. Theorem \ref{thrm: invariance of H_a under deformations} gives us a ring isomorphism $H_{\mathfrak{a}}(\f)\cong H_{\mathfrak{a}}(\f_t)$, so $H_{\mathfrak{a}}(\f_t)$ is also a Cohen--Macaulay ring, for each $t$. We now use \cite[Theorem A.6.18]{allday} to conclude that $H_{\mathfrak{a}}(\f_t)$ is a free $\sym(\mathfrak{a}^*)$-module, hence the transverse $\mathfrak{a}$-action on $\f_t$ is equivariantly formal.
\end{proof}

Let us now show the invariance of the basic Betti numbers under regular deformations when the foliation is equivariantly formal. For this, recall that if $V$ is an $\mathbb{N}$-graded vector space such that $\dim V^k$ is finite for each $k$, its \textit{Poincaré series} is the formal power series
$$\poin_{V}(s):=\sum_{k=0}^\infty (\dim V^k)s^k.$$
We will use the fact that $\poin_{V\otimes W}(s)=\poin_{V}(s)\poin_{W}(s)$, where juxtaposition denotes the Cauchy product (see, for instance, \cite[p. 14]{mccleary}).

\begin{theorem}\label{theo: Betti numbers are invariant}
Let $\f$ be a Killing foliation of a compact manifold $M$ and let $\f_t$ be a regular deformation. If $\f$ is equivariantly formal, then $b_i(\f_t)$ is constant on $t$, for each $i$.
\end{theorem}

\begin{proof}
We know from Proposition \ref{prop: formality is preserved} that the $\mathfrak{a}$-action on each $\f_t$ is equivariantly formal, that is,
$$\sym(\mathfrak{a}^*)\otimes H(\f)\cong H_{\mathfrak{a}}(\f)\cong H_{\mathfrak{a}}(\f_t)\cong\sym(\mathfrak{a}^*)\otimes H(\f_t).$$
Taking the corresponding Poincaré series we obtain
$$\poin_{\sym(\mathfrak{a}^*)}(s)\poin_{H(\f)}(s)=\poin_{\sym(\mathfrak{a}^*)}(s)\poin_{H(\f_t)}(s).$$
The factor $\poin_{\sym(\mathfrak{a}^*)}(s)$ can be canceled out on both sides since $\mathbb{Z}[[s]]$ is an integral domain, hence $\poin_{H(\f)}(s)=\poin_{H(\f_t)}(s)$ and the result follows.
\end{proof}

\begin{corollary}\label{cor: basic betti numbers as orbifold}
Let $\f$ be an equivariantly formal Killing foliation of a compact manifold $M$. Then $\f$ can be approximated by a closed foliation $\g$ such that
$$b_i(\f)=b_i(M/\!/\g)=b_i(|M/\!/\g|)$$
for each $i$.
\end{corollary}

\begin{proof}
Theorem \ref{theo: Betti numbers are invariant} implies $b_i(\f)=b_i(\g)$ for all $i$, and we have $b_i(\g)=b_i(M/\!/\g)=b_i(|M/\!/\g|)$ from Proposition \ref{prop: basic cohomology of closed foliations} and Theorem \ref{theorem: Satake}, respectively.
\end{proof}

We can now establish Theorem \ref{theorem: Gromov intro} of the introduction.

\begin{theorem}
There exists a constant $C=C(q)$ such that, for every $q$-codimensional, equivariantly formal Killing foliation $\f$ of a compact manifold $M$ with $\sec_\f>0$, one has
$$\sum_{i=0}^q b_i(\f) \leq C.$$
\end{theorem}

\begin{proof}
By Theorem \ref{theorem: deformation} we can deform $\f$ into a closed foliation $\g$ also satisfying $\sec_\g>0$. Hence $M/\!/\g$ is a compact positively curved $q$-dimensional orbifold, so $M/\g=|M/\!/\g|$ endowed with the induced Riemannian length structure is a compact positively curved $q$-dimensional Alexandrov space. By \cite[Main Theorem]{koh} (with $k=0$ and $F=\mathbb{R}$) it follows that there is a constant $C=C(q)$ such that
$$\sum_{i=0}^q b_i(\f) =\sum_{i=0}^q b_i(|M/\!/\g|)\leq C,$$
where in the first equality we use Corollary \ref{cor: basic betti numbers as orbifold}.
\end{proof}

\section*{Acknowledgements}
We are grateful to Professors Marcos M. Alexandrino, Oliver Goertsches and Hiraku Nozawa for insightful discussions. The first author also thanks the Department of Mathematics of the University of São Paulo for the welcoming environment where a substantial part of this work was developed, and FAPESP for the research funding.


\begin{thebibliography}{10}

\bibitem{adem} A. Adem, J. Leida, Y. Ruan: \textit{Orbifolds and Stringy Topology}, Tracts in Mathematics \textbf{171}, Cambridge University Press, 2007.



\bibitem{allday} C. Allday, V. Puppe: \textit{Cohomological Methods in Transformation Groups}, Cambridge Studies in Advanced Mathematics \textbf{32}, Cambridge University Press, 1993.

\bibitem{armstrong} M. Armstrong: \textit{Calculating the fundamental group of an orbit space}, Proc. Amer. Math. Soc. \textbf{84}(2) (1982), 267--271.




\bibitem{borzellino4} J. Borzellino: \textit{Orbifolds with lower Ricci curvature bounds}, Proc. Am. Math. Soc. \textbf{125}(10) (1997), 3011--3018.



\bibitem{bridson} M. Bridson, A. Haefliger: \textit{Metric spaces of non-positive curvature} Springer, 2013.

\bibitem{candel} A. Candel, L. Conlon: \textit{Foliations I}, Graduate Studies in Mathematics \textbf{23}, American Mathematical Society, 2000.


\bibitem{caramello} F. Caramello, D. Töben: \textit{Positively curved Killing foliations via deformations}, Trans. Amer. Math. Soc. \textbf{372} (2019), 8131--8158.

\bibitem{chenruan} W. Chen, Y. Ruan: \textit{Orbifold Gromov--Witten theory}, in: \textit{Orbifolds in Mathematics and Physics}, ed. by A. Adem, J. Morava and Y. Ruan, Contemporary Mathematics \textbf{310}, American Mathematical Society, 2002, 25--85.




\bibitem{galazgarcia} F. Galaz-García et al: \textit{Torus orbifolds, slice-maximal torus actions and rational ellipticity}, Int. Math. Res. Not. IMRN (2017), rnx064.

\bibitem{ghys} E. Ghys: \textit{Feuilletages riemanniens sur les varietes simplement connexes}, Ann. Inst. Fourier \textbf{34} (1984), 203--223.

\bibitem{goertsches3} O. Goertsches, H. Nozawa, D. Töben: \textit{Equivariant cohomology of $K$-contact manifolds}, Math. Ann. 354(4) (2012), 1555--1582.

\bibitem{goertsches2} O. Goertsches, H. Nozawa, D. Töben: \textit{Localization of Chern-Simons type invariants of Riemannian foliations}, Israel J. Math. \textbf{222}(2) (2017), 867--920.

\bibitem{goertsches} O. Goertsches, D. Töben: \textit{Equivariant basic cohomology of riemannian foliations}, J. Reine Angew. Math. \textbf{2018}(745) (2018), 1--40.



\bibitem{gromov2} M. Gromov: \textit{Curvature, diameter and Betti numbers}, Comment. Math. Helvetici \textbf{56} (1981), 179--195.

\bibitem{gromov} M. Gromov: \textit{Structures métriques pour les variétés riemanniennes}, Cedic-Nathan, 1981.

\bibitem{gromov3} M. Gromov: \textit{Volume and bounded cohomology}, Inst. Hautes Etudes Sci. Publ. Math. \textbf{56} (1982), 5--99.


\bibitem{guillemin}  V. Guillemin, S. Sternberg: \textit{Supersymmetry and Equivariant de Rham Theory}, Springer-Verlag, 1999.



\bibitem{haefliger2} A. Haefliger, E. Salem: \textit{Riemannian foliations on simply connected manifolds and actions of tori on orbifolds}, Illinois J. Math. \textbf{34} (1990), 706--730.



\bibitem{hebda} J. Hebda: \textit{Curvature and focal points in riemannian foliations}, Indiana Univ. Math. J. \textbf{35} (1986), 321--331.

\bibitem{kleiner} B. Kleiner, J. Lott: \textit{Geometrization of three-dimensional orbifolds via Ricci flow}, Astérisque \textbf{365} (2014), 101--177.

\bibitem{koh} L.-K. Koh: \textit{Betti numbers of Alexandrov spaces}, Proc. Amer. Math. Soc. \textbf{122}(1) (1994), 247--252.


\bibitem{lin} Y. Lin, R. Sjamaar: \textit{A Thom isomorphism in foliated de Rham theory
}, preprint arXiv:2001.11848 [math.DG].

\bibitem{loh} C. Löh: \textit{Geometric group theory}, Universitext, Springer, 2017.


\bibitem{lupercio} E. Lupercio, B. Uribe: \textit{Gerbes over orbifolds and twisted $K$-theory}, Comm. Math. Phys. \textbf{245}(3) (2004), 449--489.

\bibitem{mccleary} J. McCleary: \textit{User's Guide to Spectral Sequences}, Cambridge Studies in Advanced Mathematics \textbf{58}, Cambridge University Press, 2001.


\bibitem{milnor} J. Milnor: \textit{A note on curvature and fundamental group}, J. Differ. Geom \textbf{2}(1) (1968), 1--7.

\bibitem{mrcun} I. Moerdijk, J. Mr\v{c}un: \textit{Introduction to Foliations and Lie Groupoids}, Cambridge Studies in Advanced Mathematics \textbf{91}, Cambridge University Press, 2003.


\bibitem{molino3} P. Molino: \textit{Desingularisation des feuilletages riemanniens}, Amer. J. Math. \textbf{106}(5) (1984), 1091--1106.

\bibitem{molino} P. Molino: \textit{Riemannian Foliations}, Progress in Mathematics \textbf{73}, Birkhäuser, 1988.

\bibitem{nicolaescu} L. Nicolaescu: \textit{On a theorem of Henri Cartan concerning the equivaraint cohomology}, preprint arXiv:math/0005068 [math.DG].


\bibitem{petersen} P. Petersen: \textit{Riemannian Geometry}, Second Edition, Graduate Texts in Mathematics \textbf{171}, Springer, 2006.


\bibitem{salem} E. Salem: \textit{Riemannian foliations and pseudogroups of isometries}, Appendix D in: \textit{Riemmanian Foliations} by P. Molino, Progress in Mathematics \textbf{73}, Birkhäuser, 1988, 265--296.

\bibitem{samet} I. Samet: \textit{Betti numbers of finite volume orbifolds}, Geom. Topol. \textbf{17} (2013), 1113-1147.

\bibitem{satake} I. Satake: \textit{On a generalization of the notion of manifold}, Proc. Natl. Acad. Sci. USA,  \textbf{42}(6) (1956), 359--363.


\bibitem{serre} J.-P. Serre: \textit{Local Algebra}, Springer Monographs in Mathematics, Springer-Verlag, 2000.

\bibitem{toben} D. Töben: \textit{Localization of basic characteristic classes}, Ann. Inst. Fourier (Grenoble), \textbf{64}(2) (2014), 537--570.

\bibitem{yeroshkin} D. Yeroshkin: \textit{Riemannian orbifolds with non-negative curvature}, Doctoral Dissertation, University of Pennsylvania, 2014.

\end{thebibliography}
\end{document}